\documentclass[11pt, a4paper, oneside]{article}
\usepackage[body={16.5cm, 21.0cm},left=1.5cm,right=1.5cm]{geometry}
\usepackage{natbib}
\usepackage{amsmath, amssymb, amsthm}
\usepackage{amsfonts}
\usepackage{graphicx}
\usepackage{dsfont} 
\usepackage{color}
\usepackage{nicefrac}
\usepackage{multirow} 
\usepackage{hyperref}

\usepackage{accents}
\usepackage{setspace}
\usepackage{pifont}
\usepackage{fancyhdr}
\usepackage{soul}
\numberwithin{equation}{section}

\definecolor{mygreen}{rgb}{0.09,0.62,0.45}
\definecolor{myorange}{rgb}{0.83,0.37,0.10}
\definecolor{myblue}{rgb}{0.06,0.45,0.69}
\definecolor{myred}{rgb}{0.94,0.20,0.13}

\newtheorem{thm}{Theorem}
\newtheorem{lem}[thm]{Lemma}
\newtheorem{cor}[thm]{Corollary}

\theoremstyle{definition}
\newtheorem{rem}{Remark}

\newcommand{\keywords}[1]{{\scriptsize \noindent \textbf{KEY WORDS AND PHRASES:}} {#1}}
\newcommand{\msc}[1]{{\scriptsize \noindent \textbf{MSC2010 SUBJECT CLASSIFICATIONS:}} {#1}}

\newenvironment{pfofCor}{\noindent{\bf Proof of Corollary}}{\hfill \ding{111} \\}
\newenvironment{pfofThm}{\noindent{\bf Proof of Theorem}}{\hfill \ding{111} \\}

\newcommand{\id}{\ensuremath{\displaystyle{\mathop {=} ^d}}}

\newcommand{\field}[1]{\mathbb{#1}}
\newcommand{\real}{\ensuremath{{\field{R}}}}
\newcommand{\mc}[1]{{\ensuremath{\mathcal{#1}}}}

\newcommand{\tci}[1]{\ensuremath{c(\frac{i}{n},{#1})}}
\newcommand{\tcione}[1]{\ensuremath{c(\frac{i_1}{n},{#1})}}
\newcommand{\tcitwo}[1]{\ensuremath{c(\frac{i_2}{n},{#1})}}
\newcommand{\ci}[1]{\ensuremath{c\Bigl(\frac{i}{n},{#1}\Bigr)}}

\newcommand{\sumab}[2]{\ensuremath{\sum\limits_{#1}^{#2}}}
\newcommand{\intab}[2]{\ensuremath{\int_{#1}^{#2}}}
\newcommand{\intzero}[1]{\ensuremath{\int_{0}^{#1}}}

\newcommand{\arrowf}[1]{\ensuremath{\displaystyle {\mathop {\longrightarrow}_{#1 \rightarrow \infty}\,}}}
\newcommand{\limit}[1]{\ensuremath{\displaystyle {\lim_{#1 \rightarrow{\infty}}}}}
\newcommand{\suprem}[1]{\ensuremath{\displaystyle {\sup_{#1}}}}

\newcommand{\conv}[1]{\ensuremath{\, \displaystyle {\mathop {\longrightarrow} ^{#1}}}\, }

\newcommand{\argmax}{\mathop{\mathrm{argmax}}}

\newcommand{\ndivk}{\frac{n}{k}}

\newcommand{\one}{\mathds{1}}

\renewcommand{\mathbf}{\boldsymbol}

\newcommand{\abs}[1]{\left\vert#1\right\vert}
\newcommand{\set}[1]{\left\{#1\right\}}
\newcommand{\suit}[1]{\left(#1\right)}

\newcommand{\eps}{\varepsilon}

\title{Spatial dependence and space-time trend in extreme events}

\author{
 {John H.J. Einmahl}\\  {\small Tilburg University}\\
 \and {Ana Ferreira}\\  {\small Instituto Superior T\'{e}cnico, University of Lisbon}\\
 \and {Laurens de Haan}\\ {\small University of Lisbon}\\ {\small Erasmus University Rotterdam}\\
  \and {Cl\'{a}udia Neves}\\ {\small University of Reading}\\
  \and {Chen Zhou}\\  {\small Erasmus University Rotterdam}\\  {\small Bank of The Netherlands}\\
}%
\date{}

\begin{document}

\maketitle


\abstract{The statistical theory of extremes is extended to observations that are non-stationary and not independent. The non-stationarity over time and space is controlled via the scedasis (tail scale) in the marginal distributions. Spatial dependence stems from multivariate extreme value theory. We establish asymptotic theory for both the weighted sequential tail empirical process and the weighted tail quantile process based on all observations, taken over time and space. The results yield two statistical tests for homoscedasticity in the tail, one in space and one in time. Further, we show that the common extreme value index can be estimated via a pseudo-maximum likelihood procedure based on pooling all (non-stationary and dependent) observations. Our leading example and application is rainfall in Northern Germany.}

\vspace{0.2cm}

\vspace{0.2cm}

\msc{Primary 62G32, 62G30, 62G05, 62G10, 62G20; secondary  60F17, 60G70.}



\section{Introduction}
\label{Sec.Intro}

Within the domain of attraction of an extreme value distribution one can distinguish equivalence classes via the concept of scedasis \citep{EinmahlEtAl2016,deHaanetal2015}. The distribution function $F$ has scedasis $c$ (a positive, finite constant) with respect to the continuous distribution function $F_0$ if
\begin{equation*}
	\lim_{x\uparrow x^*}\frac{1-F(x) }{1-F_0(x)}=c,
\end{equation*}
where $x^*$ is the right endpoint of $F_0$. The equivalence class consists of all probability distributions that have a scedasis with respect to the same distribution function $F_0$.

In a temporal or a spatial context -- with independent observations -- a natural estimator of the scedasis function and its asymptotic properties are known \citep{EinmahlEtAl2016}. The present paper sets out to extend the results to a situation with dependent observations as follows. Our leading example concerns daily rainfall in Northern Germany, with measurements taken at 49 stations over 84 years. All 49 time series of daily data are split into two seasons, summer and winter, each of which comprising at least 150 days so as to stay clear of transition periods. There is evidence that the extreme value index remains constant throughout the region and time span we chose to focus on. The distribution functions of the $49\times 84 \times 150$ observations are assumed to be in one scedasis class. The distribution of the rainfall vector on any day is assumed to be in the domain of attraction of a multivariate extreme value distribution. We may assume independence in time. For the assumption of constant extreme value index and independence in time we refer to \citet{Buishand2008}, \citet{KleinTanketal2009}, and the references therein.

Consider \textit{independent} random vectors $\bigl(X _{i,1}, X _{i,2}, \ldots, X _{i,m}\bigr)$, $i=1, 2, \ldots, n$. In the rainfall context, ``$m$'' is the number of stations and ``$i$'' is time (in days). A key assumption is the existence of scedasis: for some continuous distribution function $F_0$ in the domain of attraction of an extreme value distribution
\begin{equation}\label{CondTrends}
\lim_{x\uparrow x^*}\frac{1-F_{i,j}(x)}{1-F_0(x)}= \ci{j} \in (0,\, \infty),
\end{equation}
holds for $i= 1,2,\ldots, n$, $j=1,2,\ldots,m$, where $F_{i,j}$ is the distribution function of $X_{i,j}$ and the scedasis $c(\cdot,j)$ a positive continuous function for each $j$. In order to ensure that the function $c$ is uniquely defined we impose the condition
\begin{equation*}
	\sum_{j=1}^{m}C_j(1)=1,
\end{equation*}
where for $0 \leq t\leq 1$ and $j=1,2\ldots,m$,
\begin{equation*}
	C_j(t):= \frac{1}{m}\intab{0}{t} c(u,j)\, du.	
\end{equation*}
The scedasis $c(i/n, j)$ can be interpreted as the relative frequency of extremes at time $i$ and location $j$.

We assume $F_0 \in \mc{D}(G_{\gamma})$, $\gamma \in \real$, i.e., $F_0$ is in the max-domain of attraction of $G_\gamma$. As a consequence of \eqref{CondTrends}, $\gamma$ is now the common extreme-value index: $F_{i,j} \in \mc{D}(G_{\gamma})$, $i=1, \ldots, n;\, j=1, \ldots, m$. Consider \textit{independent} random vectors $\bigl(X _{i,1}, X _{i,2}, \ldots, X _{i,m}\bigr)$, $i=1, 2, \ldots, n$. In the rainfall context, ``$m$'' is the number of stations and ``$i$'' is time (in days).

Estimators $\widehat{C}_j$ for $C_j$ will be introduced and their joint asymptotic distribution derived. This will enable us to perform various tests. For each station $j$ we test whether the scedasis is changing over time. We also test whether the $C_j(1)$ are different i.e. if there are real differences in extreme rainfall over space.

Let $F_i$  be the distribution function of $\bigl(X_{i,1}, X_{i,2}, \ldots, X _{i,m}\bigr)$. Assume that the distribution function $F_{i,j}$ of $X_{i,j} $ is continuous and let $U_{i,j}(t):= F_{i,j}^{\leftarrow}\bigl(1-1/t \bigr)$, where the arrow indicates the generalized inverse function. To model the spatial dependence, we further assume that $$\widetilde{F}(x_1, x_2, \ldots, x_m):= F_i \bigl( U_{i,1}(x_1), U_{i,2}(x_2), \ldots, U_{i,m}(x_m)\bigr)$$
does not depend on $i$ and is in the domain of attraction of a multivariate extreme value distribution \citep[Chapter 6]{deHaanFerreira2006}. As a consequence, the multivariate tail dependence structure does not depend on $i$. Let $R_{j_1,j_2}$ denote the tail copula of the components $j_1$ and $j_2$:
\begin{equation*}
	R_{j_1,j_2}(x,y)= \lim_{t \downarrow 0} \frac{1}{t} P\bigl(1-F_{i,j_1}(X_{i,j_1} ) \leq tx, \, 1-F_{i,j_2}(X_{i,j_2} ) \leq ty\bigr),\quad (x,y) \in [0,\infty]^2\backslash \{(\infty, \infty)\}.
\end{equation*}

As in \citet{EinmahlEtAl2016}, the estimator of $C_j$ could be the number of exceedances over a high empirical quantile at station $j$. But, since we want to compare the $C_j$'s we want to use the same threshold for all rain stations. Consequently the common threshold will be a high empirical quantile of all $N:= n \times m$ observations taken together. Let $X _{N-k:N}$ be the $(N-k)$-th order statistic of the observations $\bigl\{ X_{i,j}\bigr\}_{i=1,\,j=1}^{n\quad m}$.
We define the estimator
\begin{equation}\label{EstCj}
	\hat{C}_j(t):= \frac{1}{k} \sumab{i=1}{nt} \one_{\bigl\{X_{i,j}> X_{N-k:N}\bigr\}}
\end{equation}
where $k$ is an intermediate sequence i.e. $k=k(n) \rightarrow \infty$, $k(n)/n \rightarrow 0$, as $n\rightarrow \infty$.

In this paper we make  the following four contributions.
\begin{enumerate}
\item We establish the joint asymptotic behavior of $\{\hat C_j(t)\}_{t\in[0,1]}$, $j=1,\ldots,m$.
\item We test $H_0:\, C_j(1)= \frac{1}{m}\;$ for all $j=1,\ldots,m$, 	i.e., the total scedasis is constant over the various locations. We perform the test by checking whether the limit vector (in distribution) of
	\begin{equation*}
		\biggl(\sqrt{k}\Bigl(\widehat{C}_1(1)-\frac{1}{m} \Bigr),\,\sqrt{k}\Bigl(\widehat{C}_2(1)-\frac{1}{m} \Bigr),\, \ldots, \sqrt{k}\Bigl(\widehat{C}_m(1)-\frac{1}{m} \Bigr)  \biggr)	 
	\end{equation*}
	has mean zero. This will be done via an adapted $\chi^2$-test.
\item We test $H_{0,j}:\, C_j(t)= t C_j(1)\;$ for $0\leq t \leq 1$ and some given $j \in \{1,\ldots, m\}$, i.e., the scedasis $c(\cdot, j)$ is constant over time. Since, under $H_{0,j}$, the limit in distribution of the process $\bigl\{\sqrt{k} \bigl( \widehat{C}_j(t) -t\widehat{C}_j(1)\bigr\}_{0\leq t\leq 1}$ is essentially Brownian bridge, we can use, e.g., a Kolmogorov-Smirnov-type statistic.
\item We establish the asymptotic behavior of the pseudo-maximum likelihood estimator of $\gamma$ based on all $n\times m$ observations.
\end{enumerate}

Crucial for these results is a joint Gaussian approximation of the $m$ sequential  tail empirical processes as well as one for the tail quantile process based on all $n
\times m$ observations, for general $\gamma \in \real.$

The outline of the paper is as follows. Section 2 gives a detailed account of the conditions and the ensuing results. These results are applied to the mentioned rainfall data in Section 3. Proofs are collected in Section 4 and partly deferred to the supplementary material, along with a simulation study showing the performance of the proposed estimation and testing procedures.

\section{Results}\label{Sec.Results}

Throughout the paper we assume the following conditions:
\begin{itemize}
\item[\textbf{(i)}]	Spatial dependence. Assume that $\widetilde{F}(x_1, x_2, \ldots, x_m):= F_i \bigl( U_{i,1}(x_1), U_{i,2}(x_2), \ldots, U_{i,m}(x_m)\bigr)$ does not depend on $i$ and is in the domain of attraction of a multivariate extreme value distribution.
\item[\textbf{(ii)}] Sharpening of the scedasis condition \eqref{CondTrends}. Assume that there exists an eventually decreasing function $A_1$ with $\lim_{t\rightarrow\infty}A_1(t)=0$ such that
\begin{equation}\label{CondSced2ndord}
\suprem{n\geq 1}\max_{i,j}\left|\frac{1-F_{i,j}(x)}{1-F_0(x)}-c\left(\frac i n,j\right)\right|=O\left(A_1\left(\frac 1{1-F_0(x)}\right)\right), \; \mbox{ as } x \uparrow x^*.
\end{equation}
\item[\textbf{(iii)}] Second order condition for $F_0$. Write $U_0(t):= F_0^{\leftarrow}\bigl(1-1/t \bigr)$. Then there exists $\gamma \in \real$, $\rho <0$ and functions $\tilde a_0$, positive, and $A_0$ not changing sign eventually satisfying  $\lim_{t\to\infty}A_0(t)=0$ such that for all $x>0$,
\begin{equation}\label{Cond2ndord}
\limit{t}\frac{ \frac{U_0(tx)-U_0(t)}{\tilde a_0(t)}-\frac{x^\gamma-1}{\gamma} }{A_0(t)} =\Psi_{\gamma,\rho}(x),
\end{equation}
where $\Psi_{\gamma,\rho}(x)$ is as in Corollary 2.3.5 of \citet{deHaanFerreira2006}.
\item[\textbf{(iv)}] Conditions on the intermediate sequence $k$. Assume, as $n\rightarrow \infty$,
\begin{eqnarray*}
    & &k\to\infty, k/n\to 0,\\
	& &\sqrt{k} A_1\Bigl(q\ndivk \Bigr) \rightarrow 0, \mbox{ for all } q>0,\\
	& &\sqrt{k} A_0\Bigl(\ndivk \Bigr) \rightarrow 0\\
	& &\sqrt{k} \sup_{|u-v|\leq \frac{1}{n}}\bigl|c(u,j)-c(v,j) \bigr| \rightarrow 0, \quad \mbox{for } j=1,2, \ldots,m.
\end{eqnarray*}
\end{itemize}


We begin with presenting two fundamental approximations, which are the basis for the main results (Theorem \ref{ThmCj}, Corollaries \ref{CorBLoc} and \ref{CorWLoc}, Theorem \ref{mle_theor}), but they are also of independent interest. Write $\gamma_+=\gamma \vee 0$, $\gamma_-= \gamma \wedge 0$. For the following theorem we need some inequalities that hold for a different formulation of \eqref{Cond2ndord} as given in Corollary 2.3.7 of \citet{deHaanFerreira2006}. Such inequalities are valid if we replace $U_0(t)$, $\tilde a_0$ and $\Psi$ with $b_0(t)$, $a_0$ and $\overline\Psi$ therein.

\begin{thm}\label{Thm.QuantProc}
	Assume conditions (i)-(iv) with $\rho <0$. Let $x_0>-1/\gamma_+$; set $x_1:= 1/(-\gamma_-)$.
	\begin{description}
\item[a) Tail empirical distribution functions]\mbox{}\\ Using a Skorokhod construction, for $0 \leq \eta < 1/2$, as $n\rightarrow \infty$, it holds almost surely,
	\begin{equation}\label{TailEmp_j}
	\begin{split}
		 \max_{1\leq j\leq m} \suprem{x_0\leq x< x_1}	 \suprem{\, 0\leq t_j \leq 1 }\, (1+\gamma x)^{\eta/\gamma}\biggl| \sqrt{k} \biggl( \frac{1}{k} \sumab{i=1}{nt_j} \one_{\Bigl\{ \frac{X_{i,j}  -b_0( \frac{N}{k}) }{a_0( \frac{N}{k}) } > x\Bigr\}} &- (1+\gamma x)^{-1/\gamma} C_j(t_j)\biggr) \\
		& -  W_j\Bigl((1+\gamma x)^{-1/\gamma}, \,C_{j}(t_j) \Bigr)\biggr| \longrightarrow 0,
	\end{split}
	\end{equation}
where $(W_1, \ldots, W_m)$ is a Gaussian vector of bivariate Wiener processes $W_j$ with a covariance matrix $\Sigma=\Sigma(s_1,s_2, t_1,t_2)$ with entries
\begin{equation*}
	 \sigma_{j_1,j_2}(s_1,s_2,t_1,t_2):=Cov\Bigl( W_{j_1}\bigl( s_1, C_{j_1}(t_1) \bigr),\, W_{j_2}\bigl( s_2, C_{j_2}(t_2) \bigr) \Bigr)\\
	= \frac{1}{m}\intzero{t_1 \wedge t_2} R_{j_1,j_2}\Bigl(s_1 c(u,j_1),\, s_2 c(u,j_2)\Bigr) \, du,
\end{equation*}
for $1\leq j_1, j_2\leq m$.

\item[b) Tail empirical quantile function]\mbox{}\\ With $\varepsilon>0$ and $X_{1:N}\leq \ldots \leq X_{N:N}$ the order statistics of the sample $\{X_{i,j} \}_{i,j}$ of all $N$ observations, we have, for $T>0$, as $n\rightarrow \infty$,
\begin{equation}\label{QuantOverall}
	\suprem{\frac{1}{2k}\leq s \leq T} \, s^{-1/2+\varepsilon} \biggl|  s^{\gamma+1} \sqrt{k} \Bigl( \frac{X_{N-[ks]: N} -b_0( \frac{N}{k}) }{a_0( \frac{N}{k}) } - \frac{s^{-\gamma}-1}{\gamma}\Bigr) - \sumab{j=1}{m} W_j\bigl(s, C_j(1) \bigr)\biggr| \conv{P} 0.
	\end{equation}
\end{description}
\end{thm}

\begin{cor}\label{Cor.QuantProc}
Under the conditions and in the setup of Theorem \ref{Thm.QuantProc}, with $\varepsilon>0$,
\begin{equation}\label{QuantOverallCor}
	\suprem{\frac{1}{2k}\leq s \leq T} \, \bigl( 1 \wedge s^{\gamma+1/2+\varepsilon}\bigr) \biggl| \sqrt{k} \Bigl( \frac{X_{N-[ks]:N} -X_{N-k:N} }{a_0( \frac{N}{k}) } - \frac{s^{-\gamma}-1}{\gamma}\Bigr) - \sumab{j=1}{m}\Bigl( s^{-\gamma-1}W_j\bigl(s, C_j(1)\bigr)-W_j\bigl(1, C_j(1)\bigr)\Bigr)\biggr| \conv{P} 0.
\end{equation}
\end{cor}

As a result we get the joint asymptotic behavior of the $\hat{C}_j$, $j=1,\ldots,m$.

\begin{thm}\label{ThmCj}
Under the conditions and in the setup of Theorem \ref{Thm.QuantProc}, as $n\rightarrow \infty$,
\begin{equation}\label{first3}
	\max_{1\leq j\leq m} \suprem{0< t \leq 1} \,\Bigl| \sqrt{k} \bigl(\widehat{C}_j(t) - C_j(t)\bigr) - \Bigl\{W_j\bigl(1, C_j(t)\bigr) -C_j(t) \sumab{r=1}{m}W_r\bigl(1,C_r(1) \bigr) \Bigr\}  \Bigr| \conv{P} 0,
\end{equation}

Moreover, we have the uniform consistency of the estimator of the covariance matrix $\Sigma$ as follows. For $T>0$ and $j_1\neq j_2$, as $n\rightarrow \infty$,
\begin{equation}\label{CovEst}
	\suprem{0\leq t\leq 1}\,\,\,\suprem{0\leq s_1,s_2 \leq T}\, \Bigl| \frac{1}{k} \sumab{i=1}{nt} \one_{\bigl\{ X_{i,j_1}> X_{N-[ks_1]:N}, \,  X_{i,j_2}> X_{N-[ks_2]:N}\bigr\}}  -
\frac{1}{m}\intzero{t} R_{j_1,j_2}\Bigl(s_1 c(u,j_1),\, s_2 c(u,j_2)\Bigr) \Bigr| \conv{P} 0.
\end{equation}
\end{thm}


\bigskip%
Now we proceed with the two aforementioned tests. First we discuss the testing problem
\begin{equation}\label{TestBLoc}
\begin{cases}
H_{0}:C_j(1)=\frac 1m,&  \mbox{for all } j=1,\ldots,m, \\
H_{1}:C_j(1)\neq\frac 1m,& \mbox{for some } j=1,\ldots,m.
\end{cases}
\end{equation}

Let $\one_m$ be the $m$-unit vector,  $I_m$ the identity matrix of dimension $m$ and define $M:=I_m -\frac{1}{m}\one_m \one_m'$. From Theorem \ref{ThmCj} and under $H_0$,
$D= \sqrt{k}\Big(\hat{C}_1(1)-\frac{1}{m}, \ldots, \hat{C}_m(1)-\frac{1}{m} \Bigr)'$ is asymptotically an $m$-multivariate normal with zero mean vector and covariance matrix $M\Sigma_1 M'$, where $\Sigma_1=\Sigma(1,1,1,1)$. Assume that $\Sigma_1$ is invertible. Then $\text{rank}(M\Sigma_1M')=\text{rank}(M)=m-1$.
We therefore confine attention to the first $m-1$ components of $D$ denoted by $D_{m-1}$, which has an asymptotic covariance matrix $(M\Sigma_1M')_{m-1}$. Here the notation $A_{m-1}$ refers to the first $m-1$ rows and $m-1$ columns of an $m\times m$ matrix $A$, i.e. $A_{m-1}=(I_{m-1}, \mathbf{0}_{m-1})A(I_{m-1}, \mathbf{0}_{m-1})'$. Finally, we define the test statistic
\begin{equation*}
	T_n:= D_{m-1}' \left((M\hat\Sigma_1M')_{m-1}\right)^{-1} D_{m-1},	
\end{equation*}
with $\hat\Sigma_1$ estimated  via the empirical counterpart given in \eqref{CovEst}. From Theorem \ref{ThmCj} we immediately get the asymptotic behavior of $T_n$ under $H_0$.

\begin{cor}\label{CorBLoc} Assume that $\Sigma_1$ is invertible. Then  under $H_0$, $T_{n}\conv{d} \chi_{m-1}^2$, as $n\rightarrow \infty$.
\end{cor}


Next we consider, for $j\in \{1, \ldots,m\}$, the testing problem $H_{0,j}:\, C_j(t)= tC_j(1)$, $0\leq t \leq 1$, and $H_{1,j}$ that this is not the case. We can use test statistics of the Kolmogorov-Smirnov-type or Cram\'er-von Mises-type based on the process $\sqrt{k}\bigl(\hat{C}_j(t)-t \hat{C}_j(1)\bigr)\bigl/ \sqrt{\hat{C}_j(1)}$, $0 \leq t \leq 1$.
\begin{cor}\label{CorWLoc}
	Fix $j \in \{1, \ldots, m\}$. Under the hypothesis that $C_j(t)= tC_j(1)$, for $0\leq t \leq 1$,
	\begin{equation*}
	\biggl\{\sqrt{k\,\hat{C}_j(1)} \,\Bigl( \frac{\hat{C}_j(t)}{\hat{C}_j(1)}\, - \,t \Bigr)  \biggr\}_{t\in [0,1]}	 \conv{d}  \bigl\{B(t) \bigr\}_{t\in [0,1]},
	\end{equation*}
with $B$ a Brownian bridge.
\end{cor}


Finally, we introduce the maximum likelihood estimator (MLE) of $(\gamma,a_0(\frac N k))$ based on dependent and non-identically distributed observations as described in Section \ref{Sec.Intro}.
 The estimator is based on the sample $\{X_{N-i+1:N}-X _{N-k:N}\}_{i=1}^k$ for which we have the result of Corollary \ref{Cor.QuantProc}. We highlight $\gamma_0$ as the true unknown parameter value.

In particular, we know from \eqref{QuantOverallCor} that an approximate model for $\{X_{N-i+1:N}-X _{N-k:N}\}_{i=1}^k$ is the Generalized Pareto ($GP_{\gamma,\sigma}$), with  well-known log-likelihood
\begin{equation}\label{eq:ldef}
 \ell(\gamma,\sigma,x)=-\log\sigma-\left(1+\frac{1}{\gamma}\right)\log\left(1+\gamma\,\frac{x}{\sigma}\right),\quad 0<x<\frac{\sigma}{-\gamma_-},
\end{equation}
(for $\gamma=0$ the formula is interpreted as  $-\log\sigma- x/\sigma$), and standard tail quantile function,
\[(1-GP_{\gamma,1})^\leftarrow(s)=\frac{s^{-\gamma}-1}{\gamma},\quad s\in(0,1).\]
The misspecified log-likelihood based on the above sample can be written as, with parameter space $(\gamma,\sigma)\in\real\times(0,\infty)$,
\begin{eqnarray}
L_{N,k}(\gamma,\sigma_{N/k})&=&\sum_{i=1}^{k}\ell(\gamma,\sigma_{N/k},X_{N-i+1:N}-X_{N-k:N})\nonumber\\
&=&k\int_0^1\ell(\gamma,\sigma_{N/k},X_{N-[ks]:N}-X_{N-k:N})\,ds\nonumber\\
&=& k\int_0^1\ell\left(\gamma,\frac{\sigma_{N/k}}{a_0(N/k)},\frac{X_{N-[ks]:N}-X_{N-k:N}}{a_0(N/k)}\right)ds-k\log a_0(N/k).
\label{L}
\end{eqnarray}
Generally $(\hat{\gamma},\hat\sigma_{N/k})$ is an MLE if it is a local maximizer of $L_{N,k}(\gamma,\sigma_{N/k})$ solving the score equations,
\[\left\{
\begin{array}{lll}
\frac{\partial}{\partial \gamma}L_{N,k}(\gamma,\sigma_{N/k})&=&0\\
\frac{\partial}{\partial \sigma}L_{N,k}(\gamma,\sigma_{N/k})&=&0.\\
\end{array}
\right.\]

\begin{thm}\label{mle_theor}
Under conditions (i)-(iv) with $\gamma_0>-1/2$, with probability tending to 1, there exists a unique sequence of estimators $(\widehat\gamma_n,\widehat a_0(N/k))$, maximizing \eqref{L}, for which
\begin{equation*}
\sqrt{k}\left(\widehat\gamma_n-\gamma_0, \frac{\widehat a_0(N/k)}{a_0(N/k)}-1\right)\stackrel{d}\longrightarrow N(0,I_{\gamma_0}^{-1}\Sigma_{\gamma_0}I_{\gamma_0}^{-1})
\end{equation*}
where
\begin{equation*}
 		I_{\gamma_0}^{-1}= \begin{bmatrix}
 			 					(\gamma_0+1)^2 & -(\gamma_0+1)\\
 			 					-(\gamma_0+1) & 2(\gamma_0+1)
 							\end{bmatrix}	
 \end{equation*}
and $\Sigma_{\gamma_0}$ is the covariance matrix of the random vector
\begin{equation*}
\left[
\begin{array}{lc}
\sum_{j=1}^{m}\frac 1\gamma_0\int_{0}^{1}\left(s^{\gamma_0}-(1+\gamma_0)s^{2\gamma_0}\right)\left\{s^{-\gamma_0-1}W_j(s,C_j(1))-W_j(1,C_j(1))\right\} ds\\
\sum_{j=1}^{m} (1+\gamma_0)\int_{0}^{1}s^{2\gamma_0}\left\{s^{-\gamma_0-1}W_j(s,C_j(1))-W_j(1,C_j(1))\right\} ds
\end{array}\right]
\end{equation*}
with $\{W_j\}_{j=1}^m$ from Theorem \ref{Thm.QuantProc}.
\end{thm}

\begin{rem}
The covariance matrix $\Sigma_{\gamma_0}$ can be calculated as follows: let $U$ be a $(1\times 2m)$ vector with the first $i=1,\ldots,m$, components as $\gamma_0^{-1}\int_{0}^{1}\left(s^{\gamma_0}-(1+\gamma_0)s^{2\gamma_0}\right)\left\{s^{-\gamma_0-1}W_i(s,C_i(1))-W_i(1,C_i(1))\right\}ds$, and the remaining $i=m+1,\ldots,2m$ components as $(1+\gamma_0)\int_{0}^{1}s^{2\gamma_0}\left\{s^{-\gamma_0-1}W_i(s,C_i(1))-W_i(1,C_i(1))\right\}ds$. The covariance matrix of $U$ has entries $\{\tau_{ij}\}_{i=1}^{2m}$ given by,
\begin{eqnarray*}
&&\tau_{ii}=\frac{2+6\gamma_0+5\gamma_0^2}{(1+\gamma_0)^2(1+2\gamma_0)^2}\,C_i(1),\quad i=1,\ldots,m; \\
&&\tau_{ii}=\left(\frac{1+\gamma_0}{1+2\gamma_0}\right)^2C_i(1),\quad i=m+1,\ldots,2m;\\
&&\tau_{i,i+m}= \tau_{i+m,i}= \frac{1+\gamma_0}{(1+2\gamma_0)^2}C_i(1),\quad i=1,\ldots,m;\\
&&\tau_{ij}=\frac{1}{\gamma_0^2}\left\{\int_{0}^{1}\int_{0}^{1}f(s)f(t)r_{ij}(s,t)-2f(s)g(t)r_{ij}(s,1)+g(s)g(t)r_{ij}(1,1)\,ds\,dt\right\},\quad i\neq j, i,j=1,\ldots,m;\\
&&\tau_{ij}=(1+\gamma_0)^2\left\{\int_{0}^{1}\int_{0}^{1}s^{\gamma_0-1}t^{\gamma_0-1}r_{i-m,j-m}(s,t)-2s^{\gamma_0-1}t^{2\gamma_0}r_{i-m,j-m}(s,1)+s^{2\gamma_0}t^{2\gamma_0}r_{i-m,j-m}(1,1)\,ds\,dt\right\},\\
&&\hspace{9cm} i\neq j, i=m+1,\ldots,2m,j=m+1,\ldots,2m;\\
&&\hspace{-1cm}\tau_{ij}=\frac{1+\gamma_0}{\gamma_0}\left\{\int_{0}^{1}\int_{0}^{1}f(s)t^{\gamma_0-1}r_{i,j-m}(s,t)-f(s)t^{2\gamma_0}r_{i,j-m}(s,1)-g(s)t^{\gamma_0-1}r_{i,j-m}(1,t)+g(s)t^{2\gamma_0}r_{i,j-m}(1,1)\,ds\,dt\right\},\\
&&\hspace{9cm} j\neq i+m, i=1,\ldots,m,\,j=m+1,\ldots,2m; \\
&&\hspace{-1cm}\tau_{ij}=\frac{1+\gamma_0}{\gamma_0}\left\{\int_{0}^{1}\int_{0}^{1}f(s)t^{\gamma_0-1}r_{i-m,j}(s,t)-f(s)t^{2\gamma_0}r_{i-m,j}(s,1)-g(s)t^{\gamma_0-1}r_{i-m,j}(1,t)+g(s)t^{2\gamma_0}r_{i-m,j}(1,1)\,ds\,dt\right\},\\
&&\hspace{9cm}  i\neq j+m, j=1,\ldots,m,\,i=m+1,\ldots,2m,
\end{eqnarray*}
where
\[f(s)=s^{-1}-(1+\gamma_0)s^{\gamma_0-1},\,\, g(s)=s^{\gamma_0}-(1+\gamma_0)s^{2\gamma_0},\,\, \text{ and } r_{ij}(s,t)=\sigma_{i,j}(s,t,1,1)=EW_i(s,C_i(1))W_j(t,C_j(1)).
\]
Then  $\Sigma_{\gamma_0}=Cov(I\times U)=I\,Cov(U)\,I^T$ where $I=I_{2\times 2m}=\left[
\begin{array}{cccccc}
1&\cdots&1&0&\cdots&0\\
0&\cdots&0&1&\cdots&1
\end{array}\right]$.
\end{rem}
\section{Application}
\label{Sec.App}

This section is devoted to illustrating the testing methods for detecting a trend in extreme rainfalls, both across stations and over time. We use a subset of rainfall data from the German national meteorological service, which consists of daily rainfall amounts recorded in 49 stations ($m=49$) in three regions of North-West Germany: Bremen, Niedersachsen and Hamburg. The data set comprises nearly complete time series records over 84 years (1931-2014) . We divide the data into two seasons: winter from November to March, and summer from May to September.
Although the raw data set comprises a tally of $49\times 84\times 150$ rainfall amounts within each season, the actual number of observations we use at each station, $n$, will be determined by a declustering procedure. This has been designed to remove the effect of temporal dependence and is viewed as a key step to ensure that after pre-processing, the data set can be regarded as having no temporal dependence. The idea of our pre-processing procedure is to create gaps between consecutive observations by removing some days in the data set. The detailed procedure we have employed is outlined in the next paragraph.

The raw data set consists of daily rainfall amounts (in \emph{mm}) at each gauging station $j=1, \ldots, m$, including zero rainfall. From this data set we will use the daily maximum rainfall amount across the $m$ stations, henceforth referred to as station-wise maxima, for eliciting potential serial dependence. We order all station-wise maxima from high to low. The declustering procedure is initiated by picking up the pair of calendar days with the largest and second largest station-wise maxima. If this second maximum was recorded within two consecutive days of the first station-wise maximum, then all $m$ observations on its corresponding day are removed; otherwise both days are kept. This procedure then rolls out to the subsequent ordered station-wise maxima: for each station-wise maxima, we remove the corresponding day if it is recorded within two consecutive days of any of the previously kept days. This procedure results in the declustered data set used for testing the presence of scedasis over time and/or across space.

\begin{figure} 
\begin{center}
\includegraphics[scale=0.4]{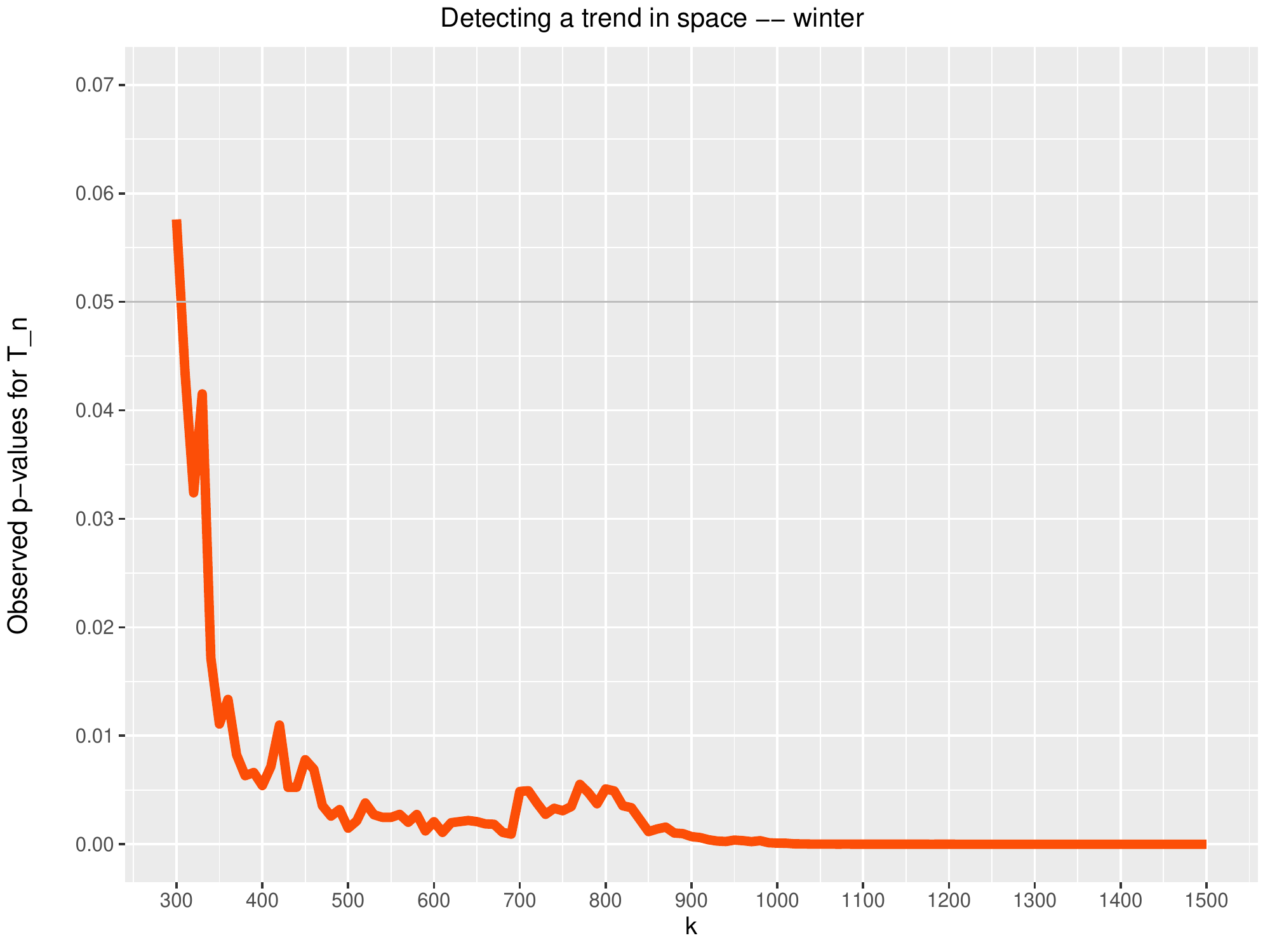}
\includegraphics[scale=0.4]{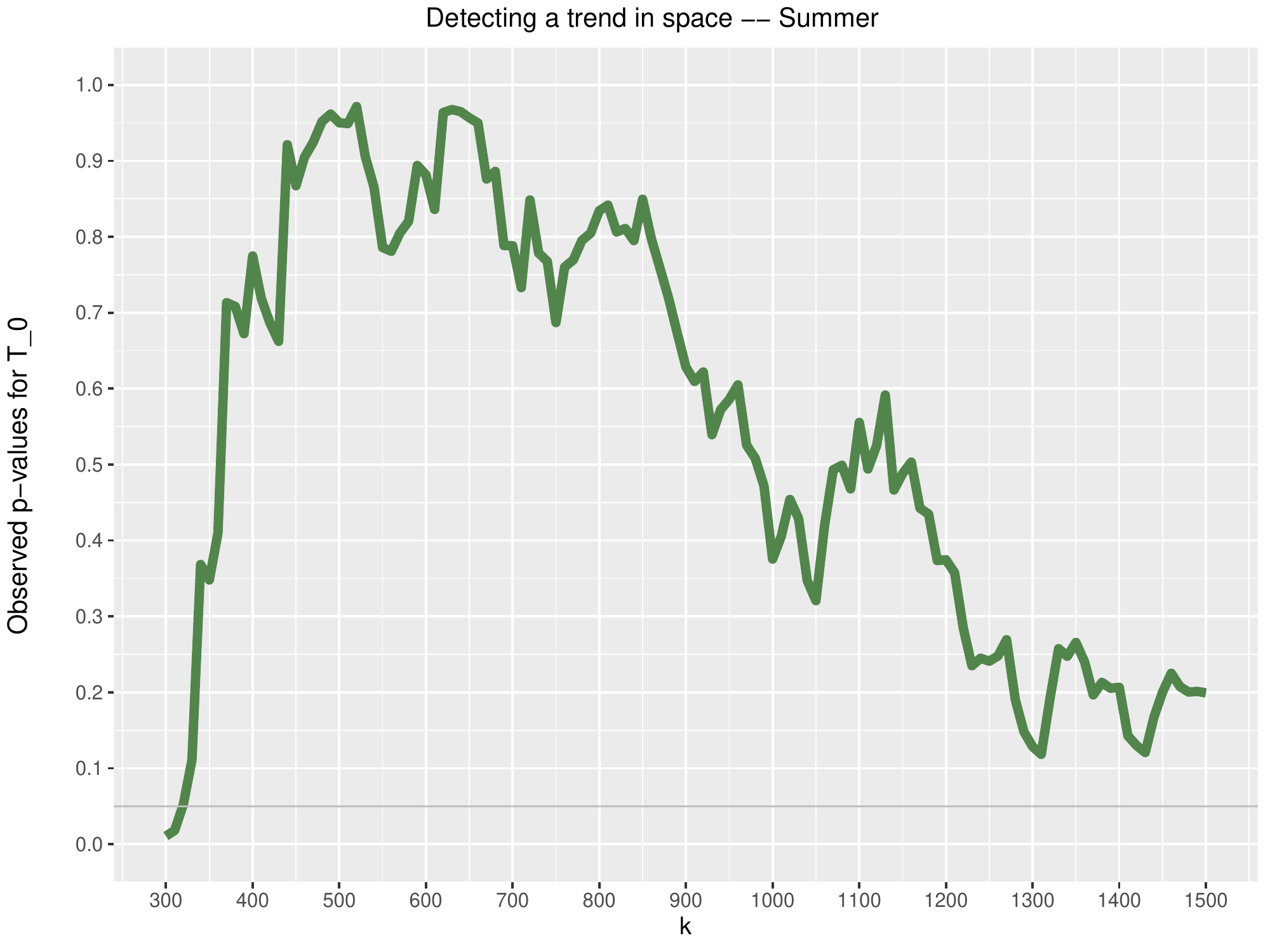}
\end{center}
\caption{\small{Obtained $p$-values through the test $T_n$ for homogeneity across space, all plotted as a function of $k= 300, \ldots, 1500.$}}
\label{Fig.T0BLoc}
\end{figure}

First, we test whether the scedasis of extreme rainfall is constant across $m=49$ stations by adopting the test statistic $T_n$ in Corollary \ref{CorBLoc}. We reject the null hypothesis of having constant scedasis across all stations for large values of $T_n$. We plot the $p$-values against $k$ the number  of upper observations used in the test in the two plots of Figure \ref{Fig.T0BLoc}, for winter and summer seasons respectively. The $p$-values obtained for the winter season stay below $5\%$ for all $k>350$. Therefore, we conclude that for the winter season, the scedasis of extreme rainfall is not constant across stations. In other words, the frequencies of having extreme rainfall differ across stations. In contrast, there is no statistical evidence of a trend in the space-domain over the summer. This finding holds for almost all values of $k$ depicted in the right panel of Figure \ref{Fig.T0BLoc}.

\begin{figure} 
\begin{center}
\includegraphics[scale=0.42]{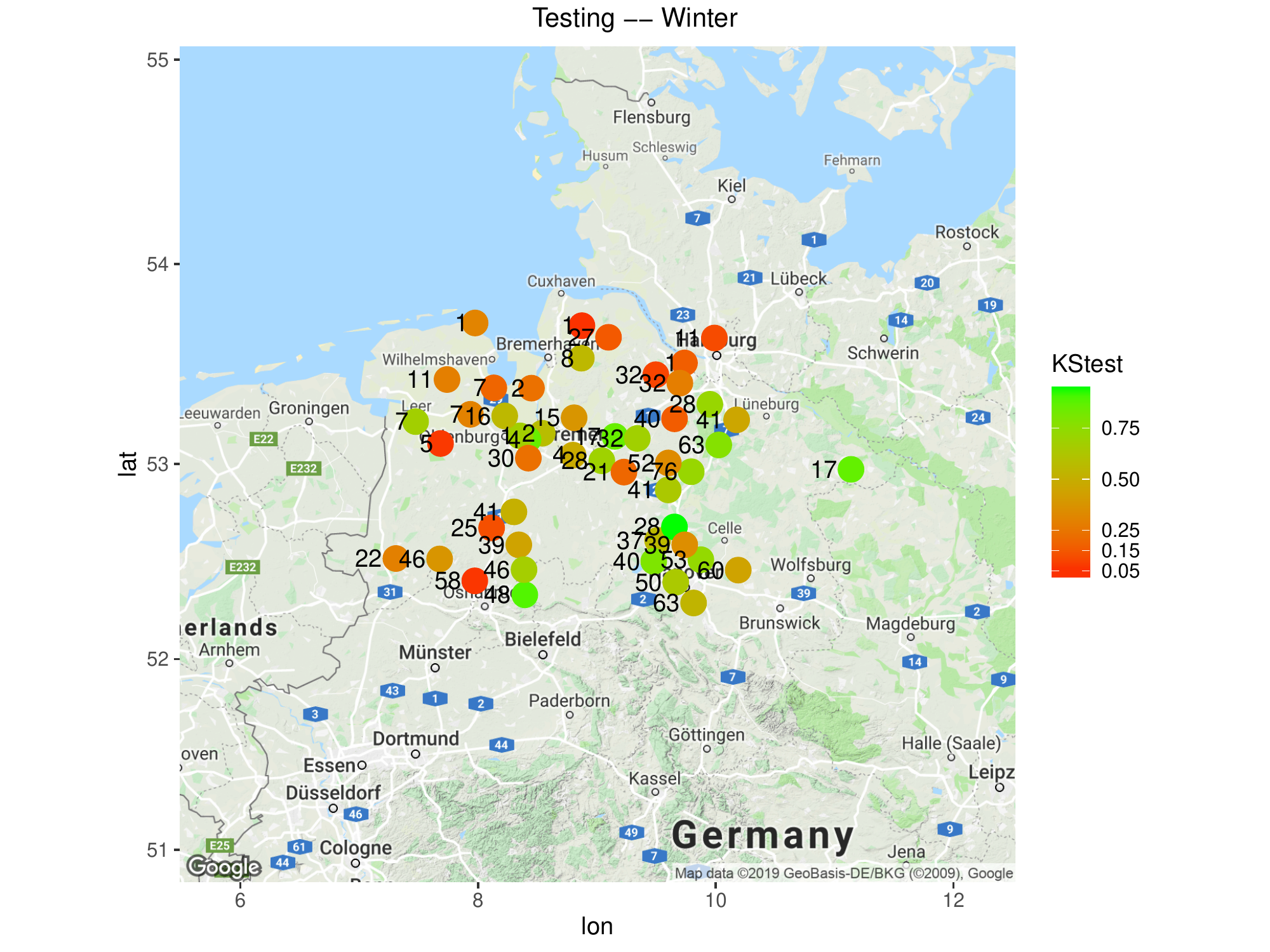}
\includegraphics[scale=0.42]{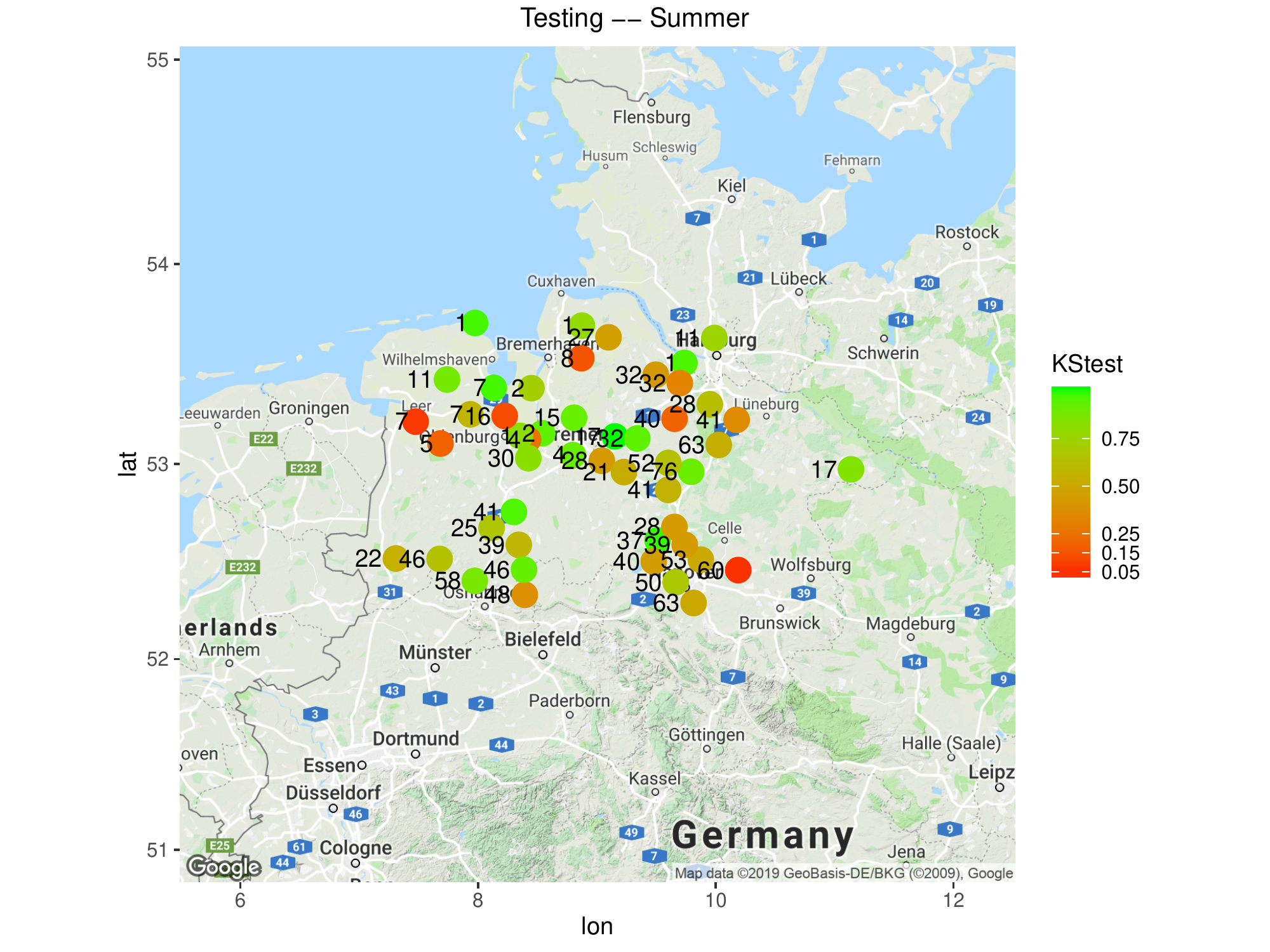}
\end{center}
\caption{\small{Obtained $p$-values for the test of the null hypothesis $H_{0,j}$ of a non-existent time trend at each station $j=1,2,\ldots,m$. The Kolmogorov-Smirnov statistic is applied with $k=1000$ higher observations. The numbers next to the station-marks indicate elevation in meters.}}
\label{Fig.TestsWLoc}
\end{figure}

Next, we investigate a possible temporal trend in the extreme rainfall process for each station mapped in Figure \ref{Fig.TestsWLoc} by means of a Kolmogorov-Smirnov (KS) type test based on the left hand side of the limit relation in Corollary \ref{CorWLoc}. For each season, we apply this test at each station $j$ with $k=1000$, and plot the $p$-values of the test in the two plots in Figure \ref{Fig.TestsWLoc}, for winter and summer seasons respectively. The sharper the red in the renderings, the lower the estimated $p$-values, and the more evidence for rejecting the null hypothesis. The brighter the green marks, the higher the $p$-values. We find that the $p$-values vary widely across the selected region, and more so in the winter.

Overall, we find that $p$-values plunge in the winter but soar in the summer at many locations. In the winter season, the KS type test highlights two stations with p-values below the nominal level $\alpha=5\%$: $p=0.044$ for station \emph{Steinau, Kr. Cuxhaven}, with elevation $1m$, and $p=0.05$ for \emph{Bramsche} at $58m$ high. Nevertheless, we need to interpret such $p$-values with caution. Given that these are the lower $p-$values across all 49 stations, we are encountering a potential multiple test problem. One potential solution is to consider the Bonferroni correction: the corrected nominal level is $\alpha^*= 5\%/49= 0.1\%$. Since these low $p$-values do not breach the corrected nominal level, there is not enough evidence in the data to reject the null hypotheses of no local trend over the winter, at the usual significance levels. Similarly, for the summer, the KS type test identifies one significant individual $p$-value of $0.048$ for station \emph{Uetze}, standing at $60m$ of elevation. Again this individual $p$-value is not in the vicinity of the Bonferroni's corrected critical barrier $\alpha^*= 0.1\%$. To summarize, there seems to be no temporal trend in extreme rainfalls in the winter or in the summer.
\begin{figure} 
\begin{center}
\includegraphics[scale=0.42]{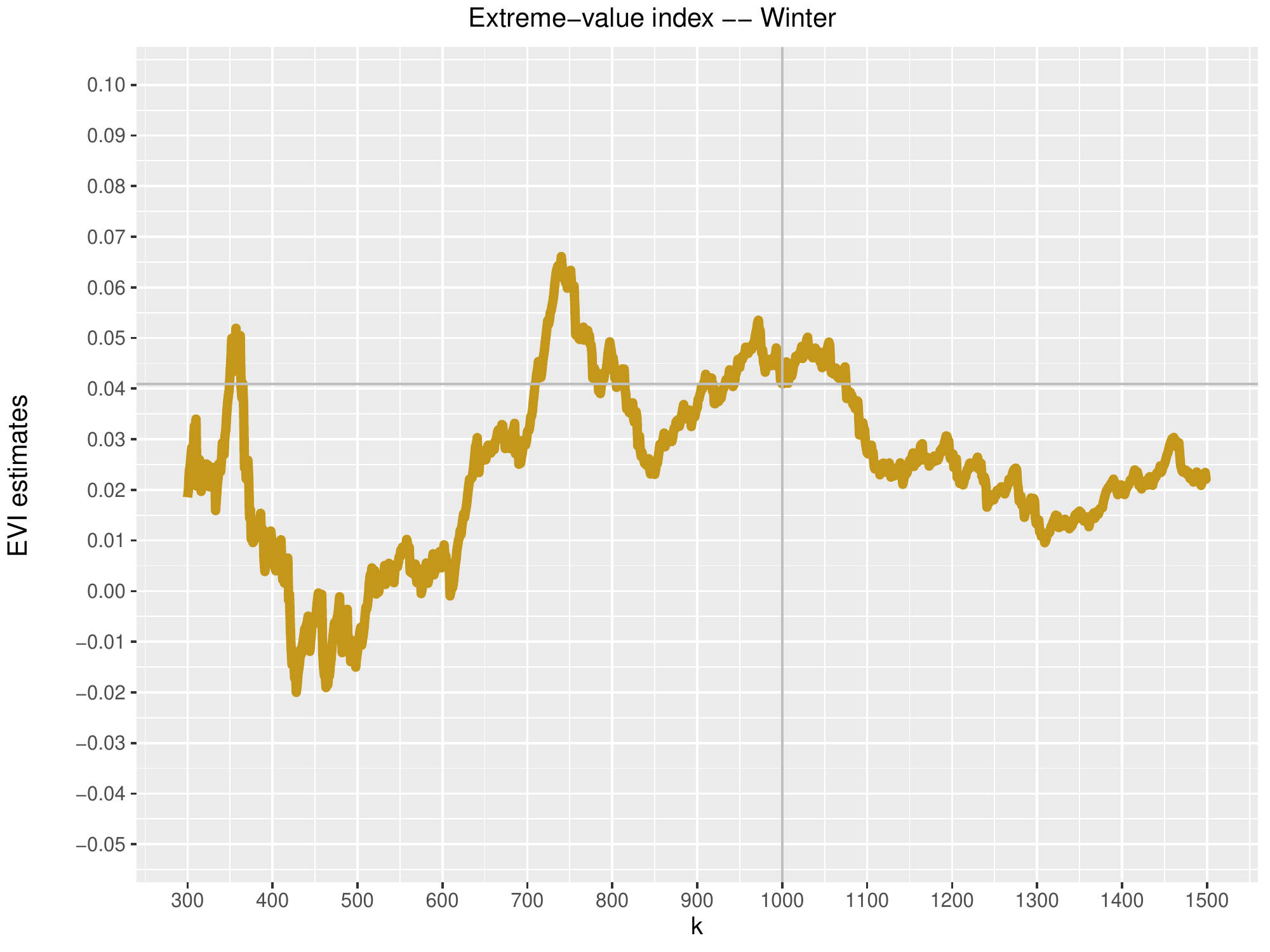}
\includegraphics[scale=0.42]{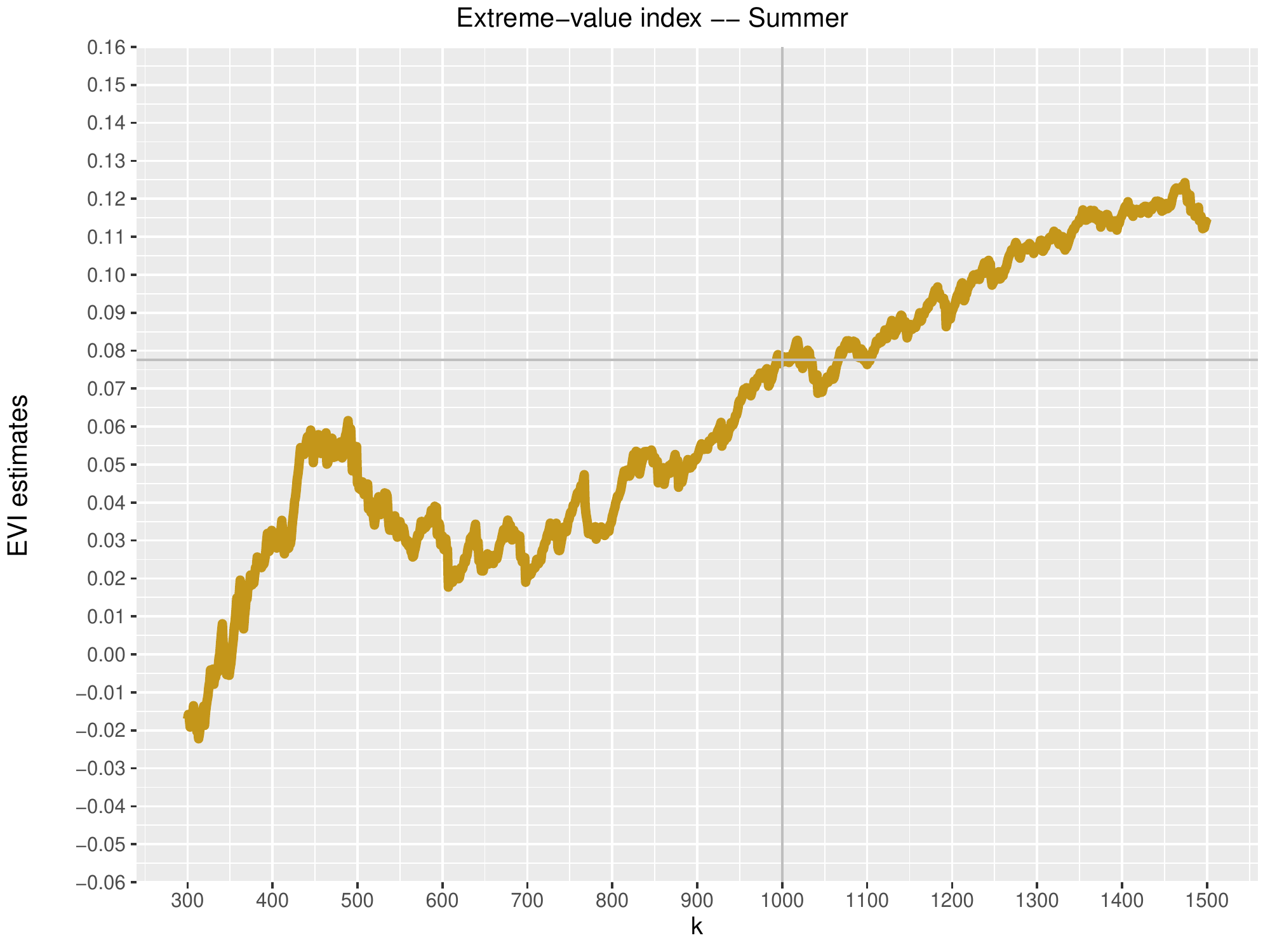}
\end{center}
\caption{\small{Sample paths of the maximum likelihood estimates of the extreme-value index, all plotted against the number $k$ of higher order statistics.}}
\label{Fig.EVIs}
\end{figure}

Finally, we report the estimated extreme value index $\gamma$ using \emph{all} data from $\emph{all}$ stations in one season, using the maximum likelihood estimator in Theorem \ref{mle_theor}. Figure \ref{Fig.EVIs} shows the estimates against various values of $k$, for the winter and summer seasons respectively. We observe that for $k$ ranging between 950 and 1100, both estimates paths  seem to consolidate a plateau of stability. For the purpose of point estimation, we fix $k=1000$, highlighted in both plots with a vertical grey line. The estimated extreme value indices are $\hat{\gamma}= 0.040$ and $0.078$ for winter and summer seasons, respectively. We conclude that the magnitude of extreme rainfalls in the summer is higher than that in the winter.

\section{Proofs}
\label{Sec.Proofs}

Write for convenience $X_{i,j}=U_{i,j}(Y_{i,j})$, where $(Y_{i,1}, Y_{i,2}, \ldots, Y_{i,m})$ follows the distribution function $\widetilde{F}$. Let $Y_{i:n}^{(j)}$ be the $i$-th order statistic from $Y_{1,j}, Y_{2,j}, \ldots, Y_{n,j}$, for all $j$.

\begin{pfofThm} \textbf{\ref{Thm.QuantProc}} {\bfseries a) Tail empirical distribution functions}\\
Consider one station $j$ for the time being and define
\begin{equation*}
	C_{j,n}(t):= \frac{1}{N} \sumab{i=1}{nt} \ci{j}.	
\end{equation*}
 According to Proposition 1 of \citet{EinmahlEtAl2016} we have under a Skorokhod construction for any $t_0>0$ and $0\leq \eta < 1/2$, almost surely,
 \begin{equation*}
 \suprem{0< v\leq t_0, \, 0\leq t\leq 1}\, v^{-\eta}\biggl|  \sqrt{k} \Bigl\{ \frac{1}{k} \sumab{i=1}{nt} \one_{\bigl\{ Y_{i,j} > \frac{n m C_{j,n}(1)}{ k v\, \tci{j}}\bigr\}} - \frac{v \,C_{j}(t)}{C_{j,n}(1)}   \Bigr\}- W_j\Bigl(v,\frac{C_{j}(t)}{C_{j,n}(1)} \Bigr) \biggr| \longrightarrow 0,
 \end{equation*}
as $n \rightarrow \infty$. After some rearrangement we get, almost surely,
\begin{equation}\label{STEP1}
 \suprem{0< v\leq t_0, \, 0\leq t\leq 1}\, v^{-\eta}\biggl|  \sqrt{k} \Bigl\{ \frac{1}{k} \sumab{i=1}{nt} \one_{\bigl\{ Y_{i,j} > \frac{n }{ k v\, \tci{j}}\bigr\}} - mv C_{j}(t)   \Bigr\}- W_j\bigl(mv, C_{j}(t) \bigr) \biggr|\longrightarrow 0.
 \end{equation}
 We are going to transform this result in several steps. First replace $v$ with $\frac{n}{ k \tci{j}} \Bigl( 1-F_{i,j}\bigl(U_0(\frac{N}{ku}) \bigr) \Bigr)$ for $0<u \leq t_0$. Note that by condition (ii), as $n \rightarrow \infty$,
\begin{equation*}
	\frac{n}{ k \tci{j}}\biggl( 1-F_{i,j}\Bigl(U_0\bigl(\frac{N}{ku}\bigr) \Bigr) \biggr)= \frac{n}{ k } \biggl( 1-F_{0}\Bigl(U_0\bigl(\frac{N}{ku}\bigr) \Bigr) \biggr) \biggl\{ 1+ O\biggl(A_1 \Bigl( \frac{N}{ku}\Bigr) \biggr)\biggr\} = \frac{u}{m} \biggl\{ 1+ O\biggl(A_1 \Bigl( \frac{N}{ku}\Bigr) \biggr)\biggr\}
\end{equation*}
and by condition (iv)
\begin{equation*}
	\sqrt{k} A_{1}\Bigl(\frac{N}{ku} \Bigr)  \rightarrow 0 \quad \mbox{uniformly for } 0<u \leq t_0.	
\end{equation*}
Hence we have, almost surely,
\begin{equation*}
	\suprem{0< u\leq t_0, \, 0\leq t\leq 1}\, u^{-\eta}\biggl|  \sqrt{k} \Bigl\{ \frac{1}{k} \sumab{i=1}{nt} \one_{\bigl\{ U_{i,j}(Y_{i,j}) >\,  U_0\bigl( \frac{N}{ku}\bigr)\bigr\}} - u\,C_{j}(t)   \Bigr\}- W_j\bigl(u,\,C_{j}(t) \bigr) \biggr| \longrightarrow 0.
\end{equation*}
Next replace $u$ with $\frac{N}{k}\Bigl(1-F_0\Bigl( b_0 \bigl(\frac{N}{k} \bigr)+ x\, a_0 \bigl(\frac{N}{k} \bigr) \Bigr) \Bigr)$ and note that by condition (iii) and Proposition 3.2 (equation 3.2) of \citet{DreesdeHaanLi06}
\begin{equation*}
	\frac{N}{k}\biggl(1-F_0\Bigl( b_0 \bigl(\frac{N}{k} \bigr)+ x\, a_0 \bigl(\frac{N}{k} \bigr) \Bigr) \biggr)= (1+\gamma x)^{-1/\gamma} \biggl\{ 1+ O\biggl(A_0 \Bigl( \frac{N}{k}\Bigr) \biggr)\biggr\}
\end{equation*}
uniformly for $x \geq x_0$ and $\sqrt{k} A_0 \bigl(\frac{N}{k} \bigr) \rightarrow 0$ by condition (iv).

\noindent Recall $X_{i,j}= U_{i,j}(Y_{i,j})$. Hence we have for each $j$ and $x_0$ larger than $1/(-\gamma_-)$,
\begin{multline}\label{STEP2}
	\suprem{x_0 \leq x < x_1, \, 0\leq t\leq 1} (1+\gamma x)^{\eta/\gamma} \biggl|   \sqrt{k} \Bigl\{ \frac{1}{k} \sumab{i=1}{nt} \one_{\bigl\{ \frac{ X_{i,j}-b_0 (\frac{N}{k} ) }{a_0 (\frac{N}{k} )}>\, x\bigr\}} - (1+\gamma x)^{-1/\gamma} \,C_{j}(t)   \Bigr\}\\- W_j\Bigl((1+\gamma x)^{-1/\gamma}, \,C_{j}(t) \Bigr) \biggr| \longrightarrow 0,
\end{multline}
which yields the weak convergence of the weighted process on the left to the weighted Wiener process on the right.

It remains to prove the joint convergence of the processes at different locations. For simplicity we do that in the context of \eqref{STEP1} (standard marginals), not \eqref{STEP2}. The general case follows similar to above. First we deal with convergence of the finite dimensional distributions. After that we consider tightness. For ease of writing we confine ourselves to the first two dimensions, i.e. $\bigl\{ (Y_{i,1}, \,Y_{i,2})\bigr\}_{i=1}^{n}$ and one point $(s_1,t_1)$ and $(s_2,t_2)$ at each dimension.
 According to Cram\`er-Wold device, we look first at all linear combinations $(x,y \in \real, \, s_1,s_2>0, \, 0< t_1, t_2 \leq 1)$
\begin{multline}\label{PPm2}
\sqrt{k}\, \biggl[x \,  \frac{1}{k} \sumab{i=1}{nt_1}\biggl( \one_{\bigl\{Y_{i,1}> \frac{n}{k s_1\, \tci{1}}\bigr\}}- P\Bigl\{Y_{i,1}>\frac{n}{k s_1\, \tci{1}}\Bigr\} \biggr)\\ + y\,\frac{1}{k} \sumab{i=1}{nt_2} \biggl( \one_{\bigl\{Y_{i,2}>  \frac{n}{k s_2\, \tci{2}}\bigr\}}- P\Bigl\{Y_{i,2}>\frac{n}{k s_1\, \tci{2}}\Bigr\} \biggr)\biggr)\biggr].
\end{multline}
For the application of Lyapunov's theorem  \citep[cf.][Theorem 27.3]{Billingsley79} it is sufficient to check the limit behavior of moments of orders 2 and 4.

\noindent The variance of
$
	\sqrt{k}\,\frac{1}{k}\sumab{i=1}{nt_1} \one_{\bigl\{Y_{i,1}> \frac{n}{k s_1\, \tci{1}}\bigr\}}
$ is
\begin{equation} \label{VarMarginal}
   \frac{1}{n} \sumab{i=1}{nt_1} \ndivk P\Bigl\{Y_{i,1}>\frac{n}{k s_1\, \tci{1}}\Bigr\}\Bigl(1- P\Bigl\{Y_{i,1}>\frac{n}{k s_1\, \tci{1}}\Bigr\} \Bigr) \\
	\arrowf{n}	  s_1mC_{1}(t_1).
\end{equation}
Next we consider the covariance of the two terms which is
\begin{equation*}
	\frac{1}{k}\sumab{i_1=1}{nt_1}\sumab{i_2=1}{nt_2} E\biggl[\Bigl(\one_{\bigl\{Y_{i_1,1}> \frac{n}{k s_1\, \tcione{1}}\bigr\}}  -P\Bigl\{Y_{i_1,1}>\frac{n}{k s_1\, \tcione{1}}\Bigr\}\Bigr) \Bigl(\one_{\bigl\{Y_{i_2,2}> \frac{n}{k s_2\, \tcitwo{2}}\bigr\}}  -P\Bigl\{Y_{i_2,2}>\frac{n}{k s_2\, \tcitwo{2}}\Bigr\}\Bigr) \biggr].
\end{equation*}
Since all products are of lower order except one (cf. \eqref{VarMarginal}), we only need to concentrate on
\begin{eqnarray*}
	& &\frac{1}{k}\sumab{i_1=1}{nt_1}\sumab{i_2=1}{nt_2} E\Bigl[\one_{\bigl\{Y_{i_1,1}> \frac{n}{k s_1\, \tcione{1}}\bigr\}} \one_{\bigl\{Y_{i_2,2}> \frac{n}{k s_2\, \tcitwo{2}}\bigr\}}\Bigr]\\
	 &= &  \frac{1}{n}\sumab{i=1}{n(t_1 \wedge t_2)} \ndivk P\Bigl\{Y_{i,1}>\frac{n}{k s_1\, \tci{1}}, \, Y_{i,2}>\frac{n}{k s_2\, \tci{2}}\Bigr\}
\end{eqnarray*}
which, on the basis of assumption (i), Theorem 6.1.5 and page 222 of \citet{deHaanFerreira2006}, converges to $\intab{0}{t_1 \wedge t_2} R_{1,2}\bigl(s_1\, c(u, j_1), s_2\, c(u, j_2)\bigr)\, du$.

The fourth moments of the terms of \eqref{PPm2} are similar. Again all terms of this expression are of lower order except
\begin{multline*}
	Q_{n,i}:= x^4 P\Bigl\{Y_{i,1}>\frac{n}{k s_1\, \tci{1}}\Bigr\} + y^4 P\Bigl\{Y_{i,2}>\frac{n}{k s_2\, \tci{2}}\Bigr\} \\
	 +(4x^3y + 4xy^3 +6x^2y^2)P\Bigl\{Y_{i,1}>\frac{n}{k s_1\, \tci{1}}, \, Y_{i,2}>\frac{n}{k s_2\, \tci{2}}\Bigr\}.
\end{multline*}
Hence the sum of the fourth moments of \eqref{PPm2} boils down, asymptotically, to
\begin{equation*}
	\sumab{i=1}{n(t_1 \wedge t_2)} \Bigl(\sqrt{k} \frac{1}{k} \Bigr)^4 Q_{n,i}= \frac{1}{k}\frac{1}{n} \sumab{i=1}{n(t_1 \wedge t_2)} \ndivk Q_{n,i}
\end{equation*}
which is of order $k^{-1}$. Recall that the sum of the second order moments tends to a constant. Hence Lyapunov's theorem applies and we have proved convergence of
\begin{equation*}
	\sqrt{k}\biggl( \frac{1}{k} \sumab{i=1}{nt_1} \one_{\bigl\{Y_{i,1}> \frac{n}{k s_1\, \tci{1}}\bigr\}}- P\Bigl\{Y_{i,1}>\frac{n}{k s_1\, \tci{1}}\Bigr\}, \,  \frac{1}{k} \sumab{i=1}{nt_2} \one_{\bigl\{Y_{i,2}>  \frac{n}{k s_2\, \tci{2}}\bigr\}}  - P\Bigl\{Y_{i,2}>\frac{n}{k s_2\, \tci{2}}\Bigr\}\biggr)
\end{equation*}
to $\Bigl(W_1\bigl( s_1, mC_1(t_1)\bigr), W_2\bigl(s_2 , mC_2(t_2) \bigr)\Bigr)$ with the covariance given as $\sigma_{1,2}(s_1,s_2,t_1,t_2)$.

\noindent Next note that by assumption (iv)
\begin{eqnarray*}
& &  x\,\frac{1}{n}\sumab{i=1}{nt_1} \ndivk P\Bigl\{Y_{i,1} >\frac{n}{k s_1\, \tci{1}}\Bigr\} + y\,\frac{1}{n}\sumab{i=1}{nt_2} \ndivk P\Bigl\{Y_{i,2}>\frac{n}{k s_2\, \tci{2}}\Bigr\}\\
&=&   x \, \frac{1}{n}\sumab{i=1}{nt_1} s_1\,\ci{1}  + y\, \frac{1}{n}\sumab{i=1}{nt_2} s_2\,\ci{2}\\
& = & x\, s_1\, m C_{1}(t_1) + y\,s_2\,m C_{2}(t_2) + o\Bigl(\frac{1}{\sqrt{k}} \Bigr).
\end{eqnarray*}
It follows that
\begin{equation}\label{Proc.mEquals2}
	\sqrt{k}\biggl( \frac{1}{k} \sumab{i=1}{nt_1} \one_{\bigl\{Y_{i,1}> \frac{n}{k s_1\, \tci{1}}\bigr\}}-  s_1 mC_1(t_1), \,  \frac{1}{k} \sumab{i=1}{nt_2} \one_{\bigl\{Y_{i,2}>  \frac{n}{k s_2\, \tci{2}}\bigr\}}  - s_2 mC_{2}(t_2)\biggr)
\end{equation}
converges to $\Bigl( W_1\bigl( s_1 , m C_{1}(t_1)\bigr),\, W_2\bigl(s_2, m C_{2}(t_2)\Bigr)$.

\noindent Next we prove tightness of the  process \eqref{Proc.mEquals2} in the space $D\bigl([0,1]^2 \times [0,T]^2\bigr)$ with index $(t_1,t_2,s_1,s_2)\in[0,1]^2\times [0,T]^2$ for $T>0$. We know from \citet{EinmahlEtAl2016} that for $0\leq \eta <1/2$ the sequence of processes
\begin{equation*}
s_j^{-\eta} \sqrt{k}\biggl( \frac{1}{k} \sumab{i=1}{nt_j} \one_{\bigl\{Y_{i,j}> \frac{n}{k s_j\, \tci{j}}\bigr\}}- s_jm C_j(t_j)\biggr)
\end{equation*}
is tight in the space $D\bigl([0,1] \times [0,T]\bigr)$, $T>0$, for $j=1$ and $j=2$. It then follows from \citet{FergerVogel2015} that the joint process \eqref{Proc.mEquals2} is also tight.

Hence the weak convergence is established. A Skorokhod construction yields the result.
\end{pfofThm}

The following lemma is useful for the proof of Theorem \ref{Thm.QuantProc} b).
\begin{lem}\label{lemma for the boundary of the quantile function}
For every $\delta>0$, there exists $0<a<1$ such that for large $n$,
\begin{equation*}
	P\Bigl\{\, U_0 \bigl( \frac{aN}{ks}\bigr) \leq X_{N-[ks]:N} \leq U_0 \bigl( \frac{N}{aks}\bigr) \; \mbox{ for } \frac{1}{2k}\leq s \leq 1 \Bigr\} 	> 1-\delta.
\end{equation*}
\end{lem}
\begin{proof}
Condition \eqref{CondTrends} implies, by inversion, that there exists $M>1$ and $t_0 >0$ such that
\begin{equation}\label{IneqsQuantfcts}
	U_0\Bigl(\frac{t}{M}\Bigr)\leq U_{i,j}(t)\leq U_0(Mt)
\end{equation}
holds for all $t\geq t_0$.

Note that from \citet[][inequality 1, page 419]{ShorackWellner1986}, we get that for every $\delta>0$ there exists $0<b<1$ such that
\begin{equation}\label{BoundsStandardj}
	P\Bigl\{\, \frac{b\,n}{ks} \leq Y^{(j)}_{n-[ks]:n} \mbox{ for } \frac{1}{2k} \leq s \leq 1\, \mbox{ and } Y^{(j)}_{n-[ks]:n} \leq \frac{n}{kbs} \; \mbox{ for }  0\leq s \leq 1; \, j=1, 2, \ldots, m \Bigr\} 	> 1-\delta/2.
\end{equation}
Note also that
\begin{equation}\label{Orderings}
	Y_{n-[ks]:n}^{(1)} \leq Y_{N-[ks]:N} \leq \max_{1\leq j\leq m} Y_{n- \bigl[\frac{ks}{m}\bigr]:n }^{(j)}
\end{equation}
Further, using \eqref{IneqsQuantfcts}, with probability tending to 1,
\begin{equation*}
	U_0\Bigl(\frac{1}{M}\,Y_{N-[ks]:N} \Bigr) \leq X_{N-[ks]:N} \leq U_0\bigl(M\,Y_{N-[ks]:N} \bigr) .
\end{equation*}
The result follows combining \eqref{IneqsQuantfcts}, \eqref{BoundsStandardj} and \eqref{Orderings}.
\end{proof}

\begin{pfofThm} \textbf{\ref{Thm.QuantProc}} {\bfseries b) Tail empirical quantile function}\mbox{}\\

We start from \eqref{TailEmp_j} in Theorem \ref{Thm.QuantProc}a). By taking $t_j=1$ and aggregating over $1\leq j\leq m$, we get that for any $0\leq\eta<1/2$, as $n\to\infty$, almost surely,

\begin{equation}\label{TailEmp}
 \sup_{x_0\leq x<x_1}(1+\gamma x)^{\eta/\gamma}\abs{\sqrt{k}\suit{\mathbb{P}_n(x)-(1+\gamma x)^{-1/\gamma}}-\sum_{j=1}^m W_j((1+\gamma x)^{-1/\gamma},C_j(1))}\longrightarrow 0,
\end{equation}
where
\[\mathbb{P}_n(x)=\frac{1}{k}\sum_{i=1}^n\sum_{j=1}^m \one_{\Bigl\{\frac{X_{i,j}-b_0\suit{\frac N k}}{a_0\suit{\frac N k}}>x \Bigr\}},\]
and $x_0>-1/\gamma_+$ and $x_1=1/(-\gamma_-)$.
As a consequence, for  $\delta_1>0$,
\begin{equation}\label{TailEmp2}
P\Bigl\{\, \sup_{x_0\leq x<x_1}(1+\gamma x)^{\eta/\gamma}\abs{\sqrt{k}\suit{\mathbb{P}_n(x)-(1+\gamma x)^{-1/\gamma}}-\sum_{j=1}^m W_j((1+\gamma x)^{-1/\gamma},C_j(1))}>\delta_1\Bigl\}\longrightarrow 0,
\end{equation}

We remark that the region $x_0\leq x<x_1$ has different implications for $\gamma\geq 0$ and $\gamma<0$. For $\gamma\geq 0$, it implies that $1+\gamma x\geq 1+\gamma x_0>0$, i.e. $1+\gamma x$ is bounded away from zero. For $\gamma<0$, $1+\gamma x>1+\gamma x_1=0$. Hence $1+\gamma x>0$ but not necessarily bounded away from zero. On the other hand, $1+\gamma x\leq 1+\gamma x_0$, i.e. $1+\gamma x$ is bounded away from $\infty$.

Then, split the range of $s$ in two subintervals, $[1/(2k),s_0]$ and $[s_0,T]$, where $s_0$ is a sufficiently low but fixed constant. The upper bound of $s_0$ will be determined throughout the proof.

For the range $[s_0,T]$ we use (\ref{TailEmp}) and Vervaat's Lemma \citep[cf. e.g. Appendix A of][]{deHaanFerreira2006} with $x_n(s):= \mathbb{P}_n(s)$ and $x_n^\leftarrow (s) := \bigl(X_{N-[ks]:N}-b_0(N/k) \bigr)/a_0(N/k)$. We then obtain the statement in (\ref
{QuantOverall}), with  the `sup'  taken over  $[s_0,T]$.

For $s\in [1/(2k), s_0]$, we first deal with the Gaussian processes term. Let $W_0$ be a univariate standard Wiener process. It is well-known (and follows from the law of the iterated logarithm) that for every $\tilde\delta>0$ there exists an $s_0$, such that
\begin{equation*}
	P\Bigl\{ \,\sup_{0<s\leq s_0} \frac{|W_0(s)|}{s^{\eta}}<\tilde\delta \Bigr\}>1-\tilde\delta.
\end{equation*}
Now $W_j(\,\cdot\, , C_j(1))\stackrel{d}{=}\sqrt{C_j(1)}W_0$ for all $j$. Hence for a given $\delta>0$, there exists an $s_0(\delta)$, such that for all $s_0\leq s_0(\delta)$,
\begin{equation}\label{wien}
	P\biggl\{\,\sup_{0<s\leq s_0}   s^{-\eta} \left|\sum_{j=1}^m W_j(s, C_j(1))\right|<\delta \biggr\}>1-\delta.
\end{equation}

Hence, we shall concentrate on proving that with probability larger than $1-\delta$, with a proper choice of $s_0$, for large $n$,
\begin{equation} \label{ineqtqtoprove}
\sup_{\frac{1}{2k}\leq s\leq s_0} s^{\gamma+\frac{1}{2}+\eps}\sqrt{k}\suit{\frac{X_{N-[ks]:N}-b_0\suit{\frac N k}}{a_0\suit{\frac N k}}-\frac{s^{-\gamma}-1}{\gamma}}\leq \delta,
\end{equation}
and
\begin{equation} \label{ineqtqtoprove2}
\inf_{\frac{1}{2k}\leq s\leq s_0} s^{\gamma+\frac{1}{2}+\eps}\sqrt{k}\suit{\frac{X_{N-[ks]:N}-b_0\suit{\frac N k}}{a_0\suit{\frac N k}}-\frac{s^{-\gamma}-1}{\gamma}}\geq -\delta.
\end{equation}

Again, we split the range of $s$ in two subintervals, $[1/(2k),t_n]$ and $(t_n,s_0]$, where $t_n$ depends only on the constant $a$ in Lemma \ref{lemma for the boundary of the quantile function} (eventually depending on $\delta$) and a sufficiently small $\xi>0$, although differently for the upper and lower bounds in \eqref{ineqtqtoprove}:
\[
\sqrt{k}t_n^{1/2+\varepsilon}:=\left\{\begin{array}{ll}
\delta\Delta_1^{-1}\text{ with } \Delta_1:=\frac{a^{-\gamma}-1}{\gamma}(1+\xi)>0, \text{for the upper bound},\\
\delta\Delta_2^{-1}\text{ with } \Delta_2:=\frac{1-a^\gamma}{\gamma}(1+\xi)>0, \text{for the lower bound}.
\end{array}\right.
\]
\emph{(A) Upper bound and $s\in[(2k)^{-1},t_n]$:} For simplicity, we assume that $A_0$ is eventually positive in the rest of the proof. Corollary 2.3.7 in \cite{deHaanFerreira2006} and Lemma \ref{lemma for the boundary of the quantile function} imply: for all $\varepsilon,\delta,\theta>0$, there exists $0<a<1$ such that for large $n$, with probability at least $1-\delta$,
\begin{eqnarray*}
\lefteqn{\frac{X_{N-[ks]:N}-b_0\suit{\frac N k}}{a_0\suit{\frac N k}}-\frac{s^{-\gamma}-1}{\gamma}}\\
&\leq& \frac{(as)^{-\gamma}-s^{-\gamma}}{\gamma}+\overline\Psi_{\gamma,\rho}\left(\frac{1}{as}\right)A_0\suit{\frac N k}+\suit{as}^{-\gamma-\rho-\theta} A_0\suit{\frac N k}\\
&\leq& s^{-\gamma}\left\{\frac{a^{-\gamma}-1}{\gamma}+s^{-\rho-\theta}K\,A_0\suit{\frac N k}\right\},
\end{eqnarray*}
for some $K>0$. Hence,
\begin{multline}
s^{\gamma+1/2+\varepsilon}\sqrt{k}\left(\frac{X_{N-[ks]:N}-U_0\suit{\frac N k}}{a_0\suit{\frac N k}}-\frac{s^{-\gamma}-1}{\gamma}\right)
\leq s^{1/2+\varepsilon}\sqrt{k}\left\{\frac{a^{-\gamma}-1}{\gamma}+s^{-\rho-\theta}K\,A_0\suit{\frac N k}\right\}\\
\leq\frac{\delta}{\Delta_1}\left\{\frac{a^{-\gamma}-1}{\gamma}+s^{-\rho-\theta}K\,A_0\suit{\frac N k}\right\}\leq \delta,
\end{multline}
for $n$ large, uniformly in $s\in[(2k)^{-1},t_n]$, since $\sup_{s\in[(2k)^{-1},\, t_n]}s^{-\rho-\theta}A_0\suit{\frac N k}\to0$ choosing $\theta<-\rho$.\\
\emph{(B) Upper bound and $s\in(t_n,s_0]$:} We prove that,  with probability at least $1-\delta$, for  large $n$,
\begin{equation}\label{Pnineqtoprove}
\mathbb{P}_n\left(\frac{s^{-\gamma}-1}{\gamma}+\frac{\delta}{\sqrt{k}}s^{-\gamma-1/2-\varepsilon}\right)\leq s,\quad \text{ for all }s\in(t_n,s_0],
\end{equation}
which implies the upper bound in \eqref{ineqtqtoprove}.

We start by giving some technical relations that provide the constants to determine an upper bound of $s_0$. Therefore note that for large enough $\eta$ there exists $\eta'$ with $1-\eta<1-\eta'<1/2+\varepsilon$, such that the following hold:
\begin{eqnarray}
\left(1+\frac{\delta\gamma}{\sqrt{k}}s^{-1/2-\varepsilon}\right)^{-1/\gamma}&<&\left(1+\frac{\delta\gamma}{\sqrt{k}}s^{-(1-\eta')}\right)^{-1/\gamma},\nonumber\\ \sup_{s\in(t_n,s_0]}\frac{s^{-(1-\eta')}}{\sqrt{k}}&\leq& \frac{t_n^{-(1-\eta')}}{\sqrt{k}}=\frac{t_n^{\eta'-1/2+\varepsilon}}{\sqrt{k}t_n^{1/2+\varepsilon}}\to 0,\nonumber\\
\left(1+\frac{\delta\gamma}{\sqrt{k}}s^{-(1-\eta')}\right)^{-1/\gamma}&\leq&1-c_I\frac{\delta}{\sqrt{k}}s^{-(1-\eta')}\text{ for some } 0<c_I\leq 1 \text{ and large } n,\nonumber\\
\left(1-\frac{\delta\gamma}{\sqrt{k}}s^{-(1-\eta')}\right)^{-1/\gamma}&\geq& 1+c_{II}\frac{\delta}{\sqrt{k}}s^{-(1-\eta')}\text{ for some } 0<c_{II}\leq 1 \text{ and large } n,\label{eta'ineq4}
\end{eqnarray}
where the last two inequalities follow from the inequalities $(1+\gamma x)^{-1/\gamma}\leq 1-c_Ix$ and $(1-\gamma x)^{-1/\gamma}\geq 1+c_{II}x$ respectively, for $0<x<\min(1,1/(-\gamma_-)$ and some $0<c_I,c_{II}\leq 1$. Take $s_0\leq (\delta c_I/(1+\delta_1))^{(\eta-\eta')^{-1}}$. 

We intend to apply \eqref{TailEmp2} with $x$ replaced by $\gamma^{-1}(s^{-\gamma}-1)+\delta k^{-1/2}
s^{-\gamma-1/2-\varepsilon}$. For this, note that,
\[
1+\gamma\left(\frac{s^{-\gamma}-1}{\gamma}+\frac{\delta}{\sqrt{k}}
s^{-\gamma-1/2-\varepsilon}\right)=s^{-\gamma}\left(1+\frac{\delta\gamma}{\sqrt{k}}s^{-1/2-\varepsilon}\right)
\]
and, for $s\in(t_n,s_0]$ and $\gamma>0$ the right-hand side is larger or equal to $s_0^{-\gamma}\left(1+\delta\gamma k^{-1/2}s_0^{-1/2-\varepsilon}\right)$, consequently bounded away from zero. For $\gamma<0$ the inequality is reversed and the expression is bounded above. Hence,
\begin{multline*}
\mathbb{P}_n\left(\frac{s^{-\gamma}-1}{\gamma}+\frac{\delta}{\sqrt{k}}s^{-\gamma-1/2-\varepsilon}\right)\\
\leq s\left(1+\frac{\delta\gamma}{\sqrt{k}}s^{-1/2-\varepsilon}\right)^{-1/\gamma}+\frac{1}{\sqrt{k}}\widetilde{W}\left(s\left(1+\frac{\delta\gamma}{\sqrt{k}}s^{-1/2-\varepsilon}\right)^{-1/\gamma}\right)+\frac{\delta_1}{\sqrt{k}}s^{\eta}\left(1+\frac{\delta\gamma}{\sqrt{k}}s^{-1/2-\varepsilon}\right)^{-\eta/\gamma}\\
\leq s\left(1-c_I\frac{\delta}{\sqrt{k}}s^{-(1-\eta')}\right)+\frac{1+\delta_1}{\sqrt{k}}s^{\eta}=s-\frac{s^{\eta}}{\sqrt{k}}\left(c_I\delta s^{\eta'-\eta}-(1+\delta_1)\right)
\end{multline*}
with $\widetilde{W}(s):=\sum_{j=1}^{m}W_j(s,C_j(1))$ and where for the second inequality we have applied \eqref{eta'ineq4}, $\eta<1/2$, \eqref{wien} and the fact that
\[\left(1+\frac{\delta\gamma}{\sqrt{k}}s^{-1/2-\varepsilon}\right)^{-1/\gamma}\leq1.
\]
It remains to check that $c_I\delta s^{\eta'-\eta}-(1+\delta_1)\geq 0$ which holds by the choice of $s_0$.\\
\emph{(C) Lower bound and $s\in[(2k)^{-1},t_n]$:} Corollary 2.3.7 in \cite{deHaanFerreira2006} and Lemma \ref{lemma for the boundary of the quantile function} imply: for all $\varepsilon,\delta,\theta>0$, there exists $0<a<1$ such that for large $n$, with probability at least $1-\delta$,
\begin{eqnarray*}
\lefteqn{\frac{X_{N-[ks]:N}-b_0\suit{\frac N k}}{a_0\suit{\frac N k}}-\frac{s^{-\gamma}-1}{\gamma}}\\
&\geq& \frac{(a/s)^\gamma-s^{-\gamma}}{\gamma}+\overline\Psi_{\gamma,\rho}\left(\frac{a}{s}\right)A_0\suit{\frac N k}-\suit{\frac{a}{s}}^{\gamma+\rho+\theta} A_0\suit{\frac N k}\\
&\geq& s^{-\gamma}\left\{\frac{a^\gamma-1}{\gamma}-s^{-\rho-\theta}K\,A_0\suit{\frac N k}\right\},\end{eqnarray*}
for some $K>0$, hence,
\begin{multline}
s^{\gamma+1/2+\varepsilon}\sqrt{k}\left(\frac{X_{N-[ks]:N}-b_0\suit{\frac N k}}{a_0\suit{\frac N k}}-\frac{s^{-\gamma}-1}{\gamma}\right)
\geq s^{1/2+\varepsilon}\sqrt{k}\left\{\frac{a^\gamma-1}{\gamma}-s^{-\rho-\theta}K\,A_0\suit{\frac N k}\right\}\\
\geq\frac{\delta}{\Delta_2}\left\{\frac{a^\gamma-1}{\gamma}-s^{-\rho-\theta}K\,A_0\suit{\frac N k}\right\}\geq -\delta,
\end{multline}
for large $n$, uniformly in $s\in[(2k)^{-1},t_n]$, since $\sup_{s\in[(2k)^{-1},\,t_n]}s^{-\rho-\theta}A_0\suit{\frac N k}\to0$ choosing $\theta<-\rho$.\\
\emph{(D) Lower bound and $s\in(t_n,s_0]$:} We prove that,  with probability at least $1-\delta$, for  large $n$,
\begin{equation}\label{Pnineqtoprove}
\mathbb{P}_n\left(\frac{s^{-\gamma}-1}{\gamma}-\frac{\delta}{\sqrt{k}}s^{-\gamma-1/2-\varepsilon}\right)\geq s+\frac 1 k,\quad \text{ for all }s\in(t_n,s_0],
\end{equation}
which implies the lower bound in \eqref{ineqtqtoprove}.
Similarly as in (B) apply \eqref{TailEmp2} with $x$ replaced by $\gamma^{-1}(s^{-\gamma}-1)-\delta k^{-1/2}
s^{-\gamma-1/2-\varepsilon}$. Note that,
\[
1+\gamma\left(\frac{s^{-\gamma}-1}{\gamma}-\frac{\delta}{\sqrt{k}}
s^{-\gamma-1/2-\varepsilon}\right)=s^{-\gamma}\left(1-\frac{\delta\gamma}{\sqrt{k}}s^{-1/2-\varepsilon}\right)
\]
and, for $s\in(t_n,s_0]$ and $\gamma>0$ the right-hand side is larger or equal to $s_0^{-\gamma}\left(1-\gamma\Delta_2\right)>0$ and consequently bounded away from zero. For $\gamma<0$ the inequality is reversed and the expression is bounded above.
Hence,
\begin{multline}
\mathbb{P}_n\left(\frac{s^{-\gamma}-1}{\gamma}-\frac{\delta}{\sqrt{k}}s^{-\gamma-1/2-\varepsilon}\right)\\ \geq s\left(1-\frac{\delta\gamma}{\sqrt{k}}s^{-1/2-\varepsilon}\right)^{-1/\gamma}+\frac{1}{\sqrt{k}}\widetilde{W}\left(s\left(1-\frac{\delta\gamma}{\sqrt{k}}s^{-1/2-\varepsilon}\right)^{-1/\gamma}\right)-\frac{\delta_1}{\sqrt{k}}s^{\eta}\left(1-\frac{\delta\gamma}{\sqrt{k}}s^{-1/2-\varepsilon}\right)^{-\eta/\gamma}\\
\geq s\left(1+c_{II}\frac{\delta}{\sqrt{k}}s^{-(1-\eta')}\right)-a^{-\eta}(1+\xi)^{-\eta/\gamma}(1+\delta_1)\frac{s^\eta}{\sqrt{k}}\\
=s+\frac{s^{\eta}}{\sqrt{k}}\left(c_{II}\delta s^{\eta'-\eta}-a^{-\eta}(1+\xi)^{-\eta/\gamma}(1+\delta_1)\right)\label{proofD}
\end{multline}
where we have used in particular \eqref{eta'ineq4}.
It remains to check that the right-hand side of \eqref{proofD} is larger or equal to $s+1/k$ which is equivalent to $s^{\eta}\sqrt{k}\left(c_{II}\delta s^{\eta'-\eta}-a^{-\eta}(1+\xi)^{-\eta/\gamma}(1+\delta_1)\right)\geq 1$. This holds by the choice of $\eta'<\eta$ and choosing $s_0\leq \delta c_{II}/((m/a)^\eta(1+\xi)^{-\eta/\gamma}(1+\delta_1))^{(\eta-\eta')^{-1}}$.

The theorem is thus proved by choosing $s_0$ with combining all aforementioned upper bounds, i.e. $s_0\leq \min(s_0(\delta), (\delta c_I/(1+\delta_1))^{(\eta-\eta')^{-1}}, \delta c_{II}/((m/a)^\eta(1+\xi)^{-\eta/\gamma}(1+\delta_1))^{(\eta-\eta')^{-1}} )$
\end{pfofThm}

\begin{pfofThm} \textbf{\ref{ThmCj}}

Fix $j\in \{1, \ldots, m\}$. Replace $x$ in \eqref{TailEmp_j} with $\Bigl( X_{N-k:N}-b_0\bigl(\frac{N}{k}\bigr) \Bigr)/a_0\bigl(\frac{N}{k}\bigr)$ and use condition (iii) jointly with Theorem 2.3.8 of \citet{deHaanFerreira2006} to get
\begin{align*}
	\sqrt{k}\biggl\{ \frac{1}{k} \sumab{i=1}{nt} \one_{\bigl\{  X_{i,j}>  X_{N-k:N}\bigr\} }
	&- \Bigl(1+\gamma\, \frac{ X_{N-k:N}- b_0\bigl(\frac{N}{k}\bigr)}{a_0\bigl(\frac{N}{k}\bigr)} \Bigr)^{-1/\gamma} C_j(t) \\
	& -W_j \biggl( \Bigl(1+\gamma\, \frac{ X_{N-k:N}- b_0\bigl(\frac{N}{k}\bigr)}{a_0\bigl(\frac{N}{k}\bigr)} \Bigr)^{-1/\gamma}, \, C_j(t)\biggr) \biggr\}\\
	&= o_p \biggl( \Bigl(1+\gamma\, \frac{ X_{N-k:N}- b_0\bigl(\frac{N}{k}\bigr)}{a_0\bigl(\frac{N}{k}\bigr)} \Bigr)^{-\eta/\gamma} \biggr).
\end{align*}
Now by \eqref{QuantOverall}
\begin{equation*}
	\sqrt{k}\frac{ X_{N-k:N}- b_0\bigl(\frac{N}{k}\bigr)}{a_0\bigl(\frac{N}{k}\bigr)} \conv{P} \sumab{j=1}{m} W_j \bigl(1, C_j(1) \bigr)
\end{equation*}
and hence\begin{equation*}
\sqrt{k} \Bigl\{ 1-\Bigl(1+\gamma\, \frac{ X_{N-k:N}- b_0\bigl(\frac{N}{k}\bigr)}{a_0\bigl(\frac{N}{k}\bigr)} \Bigr)^{-1/\gamma} \Bigr\}-\sqrt{k}\frac{X_{N-k:N}- b_0\bigl(\frac{N}{k}\bigr)}{a_0\bigl(\frac{N}{k}\bigr)}
\conv{P} 0.
\end{equation*}

\noindent Combining this we obtain
\begin{equation*}
	\sup_{0\leq t \leq 1}|\sqrt{k} \biggl(\frac{1 }{k} \sumab{i=1}{nt} \one_{\bigl\{  X_{i,j}>  X_{N-k:N}\bigr\}} -C_j(t)\biggr) 	-\{W_j	\bigl(1, C_j(t)\bigr) -  C_j(t) \sumab{r=1}{m} W_r\bigl(1, C_r(1) \bigr)\}|\conv{P} 0,
\end{equation*}
which yields \eqref{first3}.

To prove \eqref{CovEst}, we first prove the following limit relation: as $n\to\infty$,
\begin{equation}\label{Varlimitdet}
\frac{1}{k}\sum_{i=1}^{nt}P\suit{\frac{X_{i,j_1}-b_0\suit{\frac N k}}{a_0\suit{\frac N k}}>\frac{w_1^{-\gamma}-1}{\gamma},\frac{X_{i,j_2}-b_0\suit{\frac N k}}{a_0\suit{\frac N k}}>\frac{w_2^{-\gamma}-1}{\gamma}}\to \frac{1}{m}\int_0^t R_{j_1,j_2}(w_1c(u,j_1),w_2c(u,j_2))du,
\end{equation}
for fixed $(t,w_1,w_2)\in[0,1]\times [0,T]^2$. Since $X_{i,j}=U_{i,j}(Y_{i,j})$ with $Y_{i,j}$ standard Pareto distributed random variables, given any $\varepsilon>0$, for sufficiently large $n$,
\begin{align*}
&P\suit{\frac{X_{i,j_1}-b_0\suit{\frac N k}}{a_0\suit{\frac N k}}>\frac{w_1^{-\gamma}-1}{\gamma},\frac{X_{i,j_2}-b_0\suit{\frac N k}}{a_0\suit{\frac N k}}>\frac{w_2^{-\gamma}-1}{\gamma}}\\
=&P\suit{\frac{k}{n}Y_{1,j_1}> \frac{m}{\frac{N}{k}\suit{1-F_{i,j_1}\suit{b_0\suit{\frac N k}+a_0\suit{\frac N k}\frac{w_1^{-\gamma}-1}{\gamma}}}},\frac{k}{n}Y_{1,j_2}> \frac{m}{\frac{N}{k}\suit{1-F_{i,j_2}\suit{b_0\suit{\frac N k}+a_0\suit{\frac N k}\frac{w_2^{-\gamma}-1}{\gamma}}}}}\\
\leq& P\suit{\frac{k}{n}Y_{1,j_1}> \frac{m}{w_1c\suit{\frac{i}{n},j_1}(1+\varepsilon)},\frac{k}{n}Y_{1,j_2}> \frac{m}{w_2c\suit{\frac{i}{n},j_2}(1+\varepsilon)}}.
\end{align*}
The definition of $R_{j_1,j_2}$ implies that, as $n\to\infty$,
$$\frac{n}{k}P\suit{\frac{k}{n}Y_{1,j_1}> \frac{1}{v_1},\frac{k}{n}Y_{1,j_2}> \frac{1}{v_2}}\to R_{j_1,j_2}(v_1,v_2),$$
uniformly for $(v_1,v_2)\in [0,V]^2$ with any fixed $V>0$. By the continuity and boundedness of $c$ and of $R$, as $n\to\infty$,
\begin{multline}
\frac{n}{k} P\suit{\frac{k}{n}Y_{1,j_1}> \frac{m}{w_1c\suit{\frac{i}{n},j_1}(1+\varepsilon)},\frac{k}{n}Y_{1,j_2}> \frac{m}{w_2c\suit{\frac{i}{n},j_2}(1+\varepsilon)}}\\
-R_{j_1,j_2}\suit{\frac{w_1c\suit{\frac{i}{n},j_1}(1+\varepsilon)}{m},\frac{w_2c\suit{\frac{i}{n},j_2}(1+\varepsilon)}{m}}\to 0,
\end{multline}
uniformly in $i$, $w_1$ and $w_2$. Hence, as $n\to\infty$, uniformly in $(t,w_1,w_2)\in[0,1]\times [0,T]^2$,
\begin{align*}
&\frac{1}{k}\sum_{i=1}^{nt}P\suit{\frac{X_{i,j_1}-b_0\suit{\frac N k}}{a_0\suit{\frac N k}}>\frac{w_1^{-\gamma}-1}{\gamma},\frac{X_{i,j_2}-b_0\suit{\frac N k}}{a_0\suit{\frac N k}}>\frac{w_2^{-\gamma}-1}{\gamma}}\\
\leq& \frac{1}{k}\sum_{i=1}^{nt}P\suit{\frac{k}{n}Y_{1,j_1}> \frac{m}{w_1c\suit{\frac{i}{n},j_1}(1+\varepsilon)},\frac{k}{n}Y_{1,j_2}> \frac{m}{w_2c\suit{\frac{i}{n},j_2}(1+\varepsilon)}}\\
=&\int_0^t\frac{n}{k}P\suit{\frac{k}{n}Y_{1,j_1}> \frac{m}{w_1c\suit{\frac{[nu]}{n},j_1}(1+\varepsilon)},\frac{k}{n}Y_{1,j_2}> \frac{m}{w_2c\suit{\frac{[nu]}{n},j_2}(1+\varepsilon)}}du.\\
\to &\int_0^tR_{j_1,j_2}\suit{\frac{w_1c\suit{u,j_1}(1+\varepsilon)}{m},\frac{w_2c\suit{u,j_2}(1+\varepsilon)}{m}}du.
\end{align*}
A lower bound is derived similarly. Hence, by the homogeneity of the function $R$, \eqref{Varlimitdet} holds.

Next we prove for fixed $(t,w_1,w_2)\in [0,1]\times[0,T]^2$, as $n\to\infty$,
\begin{equation}\label{Varlimitrandom}
\frac{1}{k}\sum_{i=1}^{nt}\one_{\set{\frac{X_{i,j_1}-b_0\suit{\frac N k}}{a_0\suit{\frac N k}}>\frac{w_1^{-\gamma}-1}{\gamma},\frac{X_{i,j_2}-b_0\suit{\frac N k}}{a_0\suit{\frac N k}}>\frac{w_2^{-\gamma}-1}{\gamma}}}\stackrel{P}{\to} \frac{1}{m}\int_0^t R_{j_1,j_2}(w_1c(u,j_1),w_2c(u,j_2))du.
\end{equation}
We check the variance of the left hand side: using \eqref{Varlimitdet}, as $n\to\infty$,
\begin{align*}
&Var\suit{\frac{1}{k}\sum_{i=1}^{nt}\one_{\set{\frac{X_{i,j_1}-b_0\suit{\frac N k}}{a_0\suit{\frac N k}}<\frac{w_1^{-\gamma}-1}{\gamma},\frac{X_{i,j_2}-b_0\suit{\frac N k}}{a_0\suit{\frac N k}}<\frac{w_2^{-\gamma}-1}{\gamma}}}}\\
=&\sum_{i=1}^{nt}\frac{1}{k^2}Var\suit{\one_{\set{\frac{X_{i,j_1}-b_0\suit{\frac N k}}{a_0\suit{\frac N k}}<\frac{w_1^{-\gamma}-1}{\gamma},\frac{X_{i,j_2}-b_0\suit{\frac N k}}{a_0\suit{\frac N k}}<\frac{w_2^{-\gamma}-1}{\gamma}}}}\\
\leq &\sum_{i=1}^{nt}\frac{1}{k^2}P\suit{\frac{X_{i,j_1}-b_0\suit{\frac N k}}{a_0\suit{\frac N k}}<\frac{w_1^{-\gamma}-1}{\gamma},\frac{X_{i,j_2}-b_0\suit{\frac N k}}{a_0\suit{\frac N k}}<\frac{w_2^{-\gamma}-1}{\gamma}}\\
\leq &\frac{1}{km}(1+\varepsilon)\int_0^t R_{j_1,j_2}(w_1c(u,j_1),w_2c(u,j_2))du \to 0.
\end{align*}
Hence, by Chebyshev inequality, as $n\to\infty$,
\begin{multline}
\frac{1}{k}\sum_{i=1}^{nt}\one_{\set{\frac{X_{i,j_1}-b_0\suit{\frac N k}}{a_0\suit{\frac N k}}>\frac{w_1^{-\gamma}-1}{\gamma},\frac{X_{i,j_2}-b_0\suit{\frac N k}}{a_0\suit{\frac N k}}>\frac{w_2^{-\gamma}-1}{\gamma}}}\\
- \frac{1}{k}\sum_{i=1}^{nt}P\suit{\frac{X_{i,j_1}-b_0\suit{\frac N k}}{a_0\suit{\frac N k}}>\frac{w_1^{-\gamma}-1}{\gamma},\frac{X_{i,j_2}-b_0\suit{\frac N k}}{a_0\suit{\frac N k}}>\frac{w_2^{-\gamma}-1}{\gamma}}\stackrel{P}{\to}0.
\end{multline}
Then, \eqref{Varlimitdet} implies \eqref{Varlimitrandom}.

By the continuity of the right hand side of \eqref{Varlimitrandom} and the monotonicity of both sides of \eqref{Varlimitrandom} in $t$, $w_1$ and $w_2$, this result holds uniformly for $(t,w_1,w_2)\in[0,1]\times[0,T]^2$.

Finally, for fixed $(s_1,s_2)\in[0,T]^2$, by \eqref{QuantOverall}, as $n\to\infty$
$$\suit{1+\gamma\frac{X_{N-[ks_j]:N}-b_0\suit{\frac N k}}{a_0\suit{\frac N k}}}^{-1/\gamma}\stackrel{P}{\to} s_j,$$
for $j=1,2$. We can then replace $w_j$ in \eqref{Varlimitrandom} with
$\suit{1+\gamma\frac{X_{N-[ks_j]:N}-b_0\suit{\frac N k}}{a_0\suit{\frac N k}}}^{-1/\gamma},$
which yields \eqref{CovEst} for fixed $(t,s_1,s_2)\in[0,1]\times[0,T]^2$. The uniformity follows as before.
\end{pfofThm}


\begin{pfofCor} \textbf{\ref{CorWLoc}}
Under the null hypothesis, using Theorem \ref{ThmCj},  under a Skorokhod construction,
\begin{equation*}
	 \sqrt{k}\bigl( \hat{C}_j(t) - t \hat{C}_j(1)\bigr)
	= \sqrt{k}\bigl( \hat{C}_j(t) -C_j(t) \bigr) - t \sqrt{k} \bigl(\hat{C}_j(1) -C_j(1) \bigr)
\end{equation*}
converges uniformly $(0 \leq t \leq 1)$ to
\begin{align*}
 & W_j\bigl(1, C_j(t) \bigr) -C_j(t)\sumab{r=1}{m} W_r \bigl(1, C_j(1) \bigr) - t \Big[W_j \bigl(1, C_j(1) \bigr) - C_j(1) \sumab{r=1}{m} W_r \bigl( 1, C_r(1)\bigr) \Bigr]\\
 &= W_j\bigl(1, tC_j(1) \bigr) - t W_j\bigl(1, C_j(1) \bigr) \, \id \,\sqrt{C_j(1)}\, B(t).
\end{align*}
Then the result follows directly via Slutsky's theorem.
\end{pfofCor}

The proof of Theorem \ref{mle_theor} is deferred to the supplementary material. Here we only present the main steps of the proof.

We use ``local asymptotic normal theory", where the local log-likelihood and local score processes are fundamental, consisting of reparametrizations of the former with local parameter $h=(h_1,h_2)\in\mathds{R}^2$:
\begin{equation*}\label{def:gmusigmah}
\left\{\begin{array}{lll}
h_1&=&\sqrt k\left(\gamma_{N/k}-\gamma_0\right)\\
h_2&=&\sqrt k\left(\sigma_{N/k}/a_0(N/k)-1\right) \\
\end{array}\right.
\quad \Leftrightarrow\quad
\left\{\begin{array}{lll}
\gamma&=& \gamma_{N/k} = \gamma_0+ h_1/\sqrt{k}\\
\sigma&=& \sigma_{N/k} = a_0(\frac N k)(1+h_2/\sqrt{k}),
\end{array}\right.
\end{equation*}
\begin{eqnarray}
\widetilde L_{N,k}(h)&=&k\int_0^1\ell\left(\gamma_0+\frac{h_1}{\sqrt{k}},a_0(\frac N k)(1+\frac{h_2}{\sqrt{k}}),X_{N-[ks],N}-X_{N-k,N}\right)ds\nonumber\\
&=&k\int_0^1  \ell\left(\gamma_0+\frac{h_1}{\sqrt{k}},1+\frac{h_2}{\sqrt{k}},\frac{X_{N-[ks],N}-X_{N-k,N}}{a_0(\frac N k)}\right)ds-k\log a_0(\frac N k),\label{Ltilda}
\end{eqnarray}
and
\begin{equation}\label{locscoreproc}
\left\{\begin{array}{lll}
\frac{\partial \widetilde L_{N,k}}{\partial h_1}(h_1,h_2)&=&\frac{1}{\sqrt{k}}\frac{\partial L_{N,k}}{\partial \gamma}(\gamma,\sigma_{N/k})=0\\
\frac{\partial \widetilde L_{N,k}}{\partial h_2}(h_1,h_2)&=&\frac{1}{\sqrt{k}}\frac{\partial L_{N,k}}{\partial \sigma}(\gamma,\sigma_{N/k})=0.
\end{array}\right.
\end{equation}

The main steps of the proof are as follows:
\begin{description}
\item[a)] First prove that,
\begin{equation}\label{eq:Iexpansion}
\frac{\partial^2 \widetilde L_{N,k}(h)}{\partial h\partial h^T}=-I_{\gamma_0}+o_P(1),\quad I_{\gamma_0}=-\int_0^1 \frac{\partial^2 \ell}{\partial \theta \partial\theta^T}\left(\gamma_0,1,\frac{s^{-\gamma_0}-1}{\gamma_0}\right)ds,\quad \theta=(\gamma,\sigma)\in\real\times(0,\infty),
\end{equation}
uniformly in a large enough ball $H_n$ to ensure that it covers the true solution; $I_{\gamma_0}$ is the Fisher information matrix related to the approximate $GP_{\gamma_0,1}$ model
\[
I_{\gamma_0}= \left(\begin{array}{cc}
\frac{2}{1+3\gamma_0+2\gamma_0^2}&\frac{1}{1+3\gamma_0+2\gamma_0^2}\\ \frac{1}{1+3\gamma_0+2\gamma_0^2}&\frac{1}{1+2\gamma_0}
       \end{array}
 \right).
\]
$I_{\gamma_0}$ is positive definite, which implies that the local log-likelihood process is eventually strictly concave on $H_n$ with probability tending to 1.
\item[b)] Then, by integration one obtains an expansion for the local log-likelihood process (holding uniformly for $h$ in compact sets):
\begin{equation}\label{eq:locloglikexpansion}
\widetilde L_{N,k}(h)=\widetilde L_{N,k}(0)+h^T\frac{\partial \widetilde L_{N,k}}{\partial h} (0)-\frac 1 2h^TI_{\gamma_0}h+o_P(1)
\end{equation}
where
\begin{equation}\label{eq:partialDN}
\frac{\partial \widetilde L_{N,k}}{\partial h} (0)=\frac 1{\sqrt{k}} \sum_{i=1}^{k}\frac{\partial\ell}{\partial\theta}\left(\gamma_0,1,\frac{X_{N-i+1,N}-X_{N-k,N}}{a_0(\frac N k)}\right)\stackrel{d}{\to} N(0,\Sigma_{\gamma_0}).
\end{equation}
\item[c)] Finally the Argmax Theorem (van der Vaart 1998, Corollary 5.58) provides the result: let
\[\begin{array}{lll}
M_n(h)&=&\widetilde L_{N,k}(h)-\widetilde L_{N,k}(0),\\
M(h)&=&h^TN(0,\Sigma_{\gamma_0})-\frac{1}{2}h^TI_{\gamma_0}h ,\quad h\in\mathds{R}^2.
\end{array}\]
Then,
\[
\widehat h_n=\argmax _{h\in H_n} M_n(h)\stackrel{d}{\to} h=\argmax_{h\in\mathds{R}^2} M(h) \stackrel{d}{=} I_{\gamma_0}^{-1}N(0,\Sigma_{\gamma_0})
\]
provided $\widehat h_n$ is tight which holds as in Dombry and Ferreira (2018). Finally, note that
\[\widetilde L_{N,k}(\widehat h_n)=L_{N,k}\left(\gamma_0+\frac{\hat h_{n,1}}{\sqrt k}, a_0(\frac N k)(1+\frac{\hat h_{n,2}}{\sqrt k})\right) \]
and similarly for its derivatives.
\end{description}

\section*{Acknowledgements} John Einmahl holds the Arie Kapteyn Chair 2019-2022  and gratefully acknowledges the corresponding research support. Ana Ferreira was partially supported by FCT-Portugal: UID/Multi/04621/2019, UIDB/00006/2020 and SFRH/BSAB/142912/2018; IST:
P.5088. Laurens de Haan was financially supported by FCT - Funda\c{c}\~ao para a Ci\^encia e a Tecnologia, Portugal, through the projects UIDB/00006/2020 and PTDC/MAT-STA/28649/2017.
Cl\'audia Neves gratefully acknowledges support from EPSRC-UKRI Innovation Fellowship grant EP/S001263/1 and project FCT-UIDB/00006/2020.

\bibliography{RefGermanRain}

\providecommand{\noopsort}[1]{}
\begin{thebibliography}{}

\bibitem[Billingsley, 1979]{Billingsley79}
Billingsley, P. (1979).
\newblock {\em Probability and Measure}.
\newblock Wiley.

\bibitem[Buishand et~al., 2008]{Buishand2008}
Buishand, T.~A., de~Haan, L., and Zhou, C. (2008).
\newblock On spatial extremes: With application to a rainfall problem.
\newblock {\em Ann. Appl. Stat.}, 2(2):624--642.

\bibitem[Drees et~al., 2006]{DreesdeHaanLi06}
Drees, H., de~Haan, L., and Li, D. (2006).
\newblock Approximations to the tail empirical distribution function with
  application to testing extreme value conditions.
\newblock {\em Journal of Statistical Planning and Inference}, 136(10):3498 --
  3538.

\bibitem[Einmahl et~al., 2016]{EinmahlEtAl2016}
Einmahl, J. H.~J., {de Haan}, L., and Zhou, C. (2016).
\newblock Statistics of heteroscedastic extremes.
\newblock {\em J. R. Statist. Soc. B}, 78:31--51.

\bibitem[Ferger and Vogel, 2015]{FergerVogel2015}
Ferger, D. and Vogel, D. (2015).
\newblock Weak convergence of the empirical process and the rescaled empirical
  distribution function in the {S}korokhod product space.
\newblock {\em Theory Probab. Appl.}, 4(54):609--625, 2010
  (arXiv:1506.04324v1).

\bibitem[{\noopsort{Haan}}{de Haan} and Ferreira, 2006]{deHaanFerreira2006}
{\noopsort{Haan}}{de Haan}, L. and Ferreira, A. (2006).
\newblock {\em Extreme Value Theory: An Introduction}.
\newblock Springer.

\bibitem[{\noopsort{Haan}}{de Haan} et~al., 2015]{deHaanetal2015}
{\noopsort{Haan}}{de Haan}, L., Klein~Tank, A., and Neves, C. (2015).
\newblock On tail trend detection: modeling relative risk.
\newblock {\em Extremes}, 18(2):141--178.

\bibitem[Klein~Tank et~al., 2009]{KleinTanketal2009}
Klein~Tank, A., Zwiers, F.~W., and Zhang, X. (2009).
\newblock {\em Guidelines on {A}nalysis of extremes in a changing climate in
  support of informed decisions for adaptation}.
\newblock World Meteorological Organization.

\bibitem[Shorack and Wellner, 1986]{ShorackWellner1986}
Shorack, G.~R. and Wellner, J.~A. (1986).
\newblock {\em Empirical Processes with Applications to Statistics}.
\newblock John Wiley and Sons, New York.

\end{thebibliography}
\bibliographystyle{apalike}

\section*{Supporting information}
The detailed proof of Theorem \ref{mle_theor}, referenced in Section \ref{Sec.Proofs}, as well as simulations showcasing finite sample performance of the proposed methods are available with this paper at \url{https://bit.ly/3aJFM6B}.

\end{document}